\documentclass[12pt,reqno]{amsart}
\usepackage{amsmath,amssymb,amsfonts,amsthm,amscd,amstext,amsxtra,amsopn,array,url,verbatim,mathrsfs}
\usepackage{graphicx}
\usepackage{amsmath,amssymb,amsfonts,amsthm,amssymb,amscd,url,amstext,amsxtra,amsopn}
\usepackage{verbatim}
\usepackage{fullpage}
\usepackage{times}

\makeatletter
\@namedef{subjclassname@2010}{%
  \textup{2010} Mathematics Subject Classification}
\makeatother

\newtheorem{theorem}{Theorem}[section]
\newtheorem{lemma}[theorem]{Lemma}
\newtheorem{prop}[theorem]{Proposition}
\newtheorem{cor}[theorem]{Corollary}
\newtheorem{conj}[theorem]{Conjecture}
\theoremstyle{definition}
\newtheorem{definition}[theorem]{Definition}
\numberwithin{equation}{section}

\renewcommand{\mod}[1]{{\ifmmode\text{\rm\ (mod~$#1$)}\else\discretionary{}{}{\hbox{ }}\rm(mod~$#1$)\fi}}

\renewcommand{\L}{{\mathcal L}}

\newcommand{\piup}{\pi}

\begin{document}

\title{Limiting distributions of the classical error terms of prime number theory}

\author[Amir Akbary]{Amir Akbary}
\address{University of Lethbridge, Department of Mathematics and Computer Science, 4401 University Drive, Lethbridge, AB, T1K 3M4, Canada }
\email{amir.akbary@uleth.ca}
\author[Nathan Ng]{Nathan Ng}
\address{University of Lethbridge, Department of Mathematics and Computer Science, 4401 University Drive, Lethbridge, AB, T1K 3M4, Canada }
\email{nathan.ng@uleth.ca}
\author[Majid Shahabi]{Majid Shahabi}
\address{University of Lethbridge, Department of Mathematics and Computer Science, 4401 University Drive, Lethbridge, AB, T1K 3M4, Canada }
\email{m.shahabi@uleth.ca}


\thanks{Research for this article  is partially supported by NSERC Discovery grants.}


\begin{abstract}
Let $\phi : [0,\infty) \to \mathbb{R}$ and let $y_0$ be a non-negative constant.  
Let $(\lambda_n)_{n \in \mathbb{N}}$ be a non-decreasing sequence of positive numbers which tends to infinity, let $(r_n)_{n \in \mathbb{N}}$ be a complex sequence,  and $c$ a real number.   Assume that $\phi$ is square-integrable on $[0, y_0]$ and for $y \ge y_0$,  $\phi$ can be expressed as 
\begin{equation*}
\phi(y) = c + \Re \Big( \sum_{\lambda_n \le X}r_ne^{i\lambda_ny} \Big) + \mathcal{E}(y,X) ,
\end{equation*}
for any $X \ge X_0 >0$ where $\mathcal{E}(y,X)$ satisfies 
\begin{equation*}
\lim_{Y\to\infty}\dfrac{1}{Y}\int_{y_0}^Y|\mathcal{E}(y,e^Y)|^2dy=0.
\end{equation*}
We prove that, under certain assumptions on the exponents $\lambda_n$ and the coefficients $r_n$, $\phi(y)$ is a $B^2$-almost periodic function and thus possesses a 
limiting distribution. Also if $\{ \lambda_n \}_{n \in \mathbb{N}}$ is linearly independent over $\mathbb{Q}$, we  
explicitly calculate the Fourier transform of the limiting distribution measure. Moreover, we prove general versions of the above 
results for vector-valued functions. 
Finally, we illustrate some applications of our general theorems 
by applying them to  several classical error 
terms which occur in prime number theory.  Examples include the error term in the prime number theorem 
for an automorphic $L$-function, 
weighted sums of the M\"{o}bius function,
weighted sums of the Liouville function,
the sum of the M\"{o}bius function in an arithmetic progression, and the error term in Chebotarev's density theorem. 
\end{abstract}

\maketitle

\section{Introduction}

In recent years, limiting distributions have played a prominent role in many problems in analytic number theory.
Indeed it is convenient to study number theoretic questions from a probabilistic point of view. 
Limiting distributions have been a useful tool in problems concerning summatory functions \cite{H}, \cite{N}, prime number races \cite{RS}, \cite{FM}, \cite{La2},
and the distribution of values of $L$-functions \cite{Hejhal}, \cite{GS}, \cite{La1}. 
In this article, we shall investigate the limiting distributions associated to some of the classical error terms that occur in prime number theory. 
In 1935, Wintner \cite{Wi1} proved, assuming the Riemann hypothesis (RH), that the function
\begin{equation}
  \label{eq:Wild}
   e^{-y/2} \big(  \psi(e^y)- e^{y} \big) 
\end{equation}
possesses a limiting distribution, where $\psi(x) = \sum_{p^m \le x} \log p$. 
By his method, one may show that on RH
\begin{equation}\label{119}
 ye^{-y/2} \big( \pi(e^y) - \mathrm{Li} (e^y) \big)
\end{equation}
possesses a limiting distribution, where 
$\pi(x)=\sharp \{p\leq x\mid p\mbox{ is a prime}\}$ and
$\mathrm{Li}(x)=\int_2^x \frac{dt}{\log t}$. 
Over the years, other researchers have investigated similar questions for related error terms. 
Let $q >2$ and $a_1, \ldots, a_r$ be reduced residues modulo $q$.  Define $\pi(x;q,a)$ to be the number of primes less than or equal to $x$
which are congruent to $a$ modulo $q$. In 1994, Rubinstein and Sarnak \cite{RS} proved, assuming the generalized Riemann hypothesis for Dirichlet $L$-functions, that the vector-valued function
\begin{equation}
   \label{eq:RSld}
   y e^{-y/2}  \big(
   \varphi(q) \pi(e^y;q,a_1)- \pi(e^y), \ldots,   \varphi(q) \pi(e^y;q,a_r)- \pi(e^y)
  \big)
\end{equation}
possesses a limiting distribution. 
These distributions were employed to give a conditional solution to an old problem known as the Shanks-R\'{e}nyi prime number race game.
In 2004, Ng \cite{N} studied the sum of the M\"{o}bius function.
This arithmetic function is defined by
\begin{equation*}
 \mu(n) = \left\{ \begin{array}{ll}
                  1 & \mbox{if $n=1$,} \\
                  0 & \mbox{if $n$ is not squarefree,} \\
                  (-1)^{k} & \mbox{if $n$ is squarefree
                  and $n= p_{1} \ldots p_{k}$,} 
                  \end{array}
          \right.   
\end{equation*}
and its summatory function is   $M(x) = \sum_{n \le x} \mu(n)$. 
He showed that 
\begin{equation}
   \label{eq:Ngld}
    e^{-y/2} M(e^y)
\end{equation}
possesses a limiting distribution assuming the Riemann hypothesis and the conjectural bound 
\begin{equation*} \label{336}
   \sum_{0 < |\Im(\rho)| \le T} |\zeta'(\rho)|^{-2}  \ll T,
\end{equation*}
where $\zeta(s)$ is the Riemann zeta function and $\rho$ ranges through its non-trivial zeros. 
The common element in the proofs of the existence of a limiting distribution of \eqref{eq:Wild}, \eqref{119}, \eqref{eq:RSld}, and \eqref{eq:Ngld}
is an ``explicit formula''  for each of these functions.  For instance, the truncated explicit formula for $\psi(x)$ is 
\begin{equation*}
   \label{psiexplicit}
  \psi(x) = x - \sum_{\substack{ \zeta(\rho)=0 \\ |\Im(\rho)| \le X}} \frac{x^{\rho}}{\rho} + O \left( \frac{x \log^{2}(xX)}{X} +  
  \log x \right),
\end{equation*}
valid for $x \ge 2$ and $X > 1$ (see \cite[Chapter 17]{Da}). 
On the Riemann hypothesis, it follows that 
\begin{equation}
  \label{psiexplicit2}
  e^{-y/2} \big(  \psi(e^y)- e^{y} \big)= \Re \Bigg( \sum_{\substack{ \rho=\frac{1}{2}+i \gamma \\ 0 < \gamma \le X}} \frac{-2 e^{i y \gamma}}{\rho} \Bigg) + 
O \left( \frac{e^{\frac{y}{2}} \log^{2}(e^y X)}{X} +  y e^{-\frac{y}{2}} \right). 
\end{equation}
Based on this formula Wintner deduced that \eqref{eq:Wild} possesses a limiting distribution. 
In this article, we shall prove a general limiting distribution theorem 
for functions $\phi(y)$, possessing an explicit formula of a particular shape
which is modelled on \eqref{psiexplicit2}.  Our theorem will include the above results as special cases
and we will provide some new examples of functions with limiting distributions.   

We now recall the definition of a limiting distribution for a vector-valued function $\vec{\phi}: [0,\infty) \to \mathbb{R}^\ell$, where $\ell \in \mathbb{N}$.
\begin{definition}\label{4}
We say that a function $\vec{\phi}: [0,\infty) \to \mathbb{R}^\ell$ has a \emph{limiting distribution} $\mu$ on $\mathbb{R}^\ell$ if $\mu$ is a probability measure on $\mathbb{R}^\ell$ and 
\begin{equation*}\label{68}
\lim_{Y\rightarrow\infty}\dfrac{1}{Y}\int_0^Yf\big(\vec{\phi}(y)\big)dy=\int_{\mathbb{R}^\ell} fd\mu
\end{equation*}
for all bounded continuous real functions $f$ on $\mathbb{R}^\ell$.
\end{definition}

We next describe the functions considered in this article. Let $\phi: [0,\infty)\to\mathbb{R}$
and let $y_0$ be a non-negative constant such that $\phi$ is square-integrable on $[0, y_0]$. We shall assume 
there exists  
 $(\lambda_n)_{n \in \mathbb{N}}$, a non-decreasing sequence of positive numbers which tends to infinity, $(r_n)_{n \in \mathbb{N}}$, a complex sequence, and $c$ a real constant  such that for $y \ge y_0$
 \begin{equation} \label{eq:phi}
\phi(y) = c + \Re \Big( \sum_{\lambda_n \le X}r_ne^{i\lambda_ny} \Big) + \mathcal{E}(y,X) ,
\end{equation}
for any $X \ge X_0>0$
where $\mathcal{E}(y,X)$ satisfies 
\begin{equation} \label{9}
\lim_{Y\to\infty}\dfrac{1}{Y}\int_{y_0}^Y|\mathcal{E}(y,e^Y)|^2dy=0.
\end{equation}
\noindent There shall be various conditions imposed on the coefficients $r_n$ and the exponents $\lambda_n$. 

Our approach in proving the limiting distribution of $\phi(y)$ is to show that $\phi(y)$ is a $B^2$-almost periodic function. We say that the real function $\phi(y)$ is a $B^2$-\emph{almost periodic} function if for any $\epsilon>0$ there exists a real-valued trigonometric polynomial $$P_{N(\epsilon)}(y)=\sum_{n=1}^{N(\epsilon)} r_n(\epsilon) e^{i\lambda_n(\epsilon) y}$$ such that
$$\limsup_{Y\rightarrow \infty} \frac{1}{Y} \int_{0}^{Y} |\phi(y)-P_{N(\epsilon)}(y)|^2 dy <\epsilon^2.$$
 
Our main result is the following.  
\begin{theorem}
  \label{lim-dis}
Let $\phi: [0,\infty) \to \mathbb{R}$ satisfy \eqref{eq:phi} and \eqref{9}. 
Let $\alpha,\beta>0$, and $\gamma\geq 0$. Assume either of the following conditions:\\
(a) $\beta>1/2$ and
\begin{equation}\label{SI}
\sum_{T < \lambda_n \le T+1}|r_n|\ll\dfrac{(\log T)^\gamma}{T^\beta}
\end{equation}
for $T>0$. \\
(b) $\beta\leq \min\{1, \alpha\}$, $\alpha^2+\alpha/2<\beta^2+\beta$, and
\begin{equation}\label{AI}
\sum_{S <\lambda_n \le T}|r_n|\ll\dfrac{(T-S)^\alpha(\log T)^\gamma}{S^\beta}
\end{equation}
for $T>S>0$.\\
Then $\phi(y)$ is a $B^2$-almost periodic function and therefore possesses a limiting distribution.

\end{theorem} 
\noindent   In Theorem \ref{lim-dis}, we  prove that the conditions on $\phi$ imply that it
is a $B^2$-almost periodic function. However, as it is known that $B^2$-almost periodic functions possess limiting distributions
(see \cite[Theorem 8.3]{Wi2a} and Theorem \ref{340} in this article), we also obtain that $\phi$ possesses a limiting distribution. 
It would be interesting to determine the weakest conditions on the coefficients $(r_n)_{n \in \mathbb{N}}$ and the exponents $(\lambda_n)_{n \in \mathbb{N}}$ which imply that $\phi$ is $B^2$-almost periodic. 

Note that in part (b), the conditions $\beta \le \alpha$ and $\alpha^2+\alpha/2<\beta^2+\beta$ are equivalent to
\begin{equation}
  \label{alphacond}
  \beta \le \alpha < \sqrt{\beta^2+\beta+\frac{1}{16}}-\frac{1}{4}.
\end{equation}
The next corollary provides simpler criteria for which $\phi$ possesses a limiting distribution.
\begin{cor}\label{11}
Let $\phi: [0,\infty) \to \mathbb{R}$ satisfy \eqref{eq:phi} and \eqref{9}. 

Assume either of the following conditions:\\
(a) $r_n \ll \lambda_n^{-\beta}$ for $\beta>\frac{1}{2}$, and 
\begin{equation}\label{lambda}
\sum_{T<\lambda_n \le T+1}1\ll\log T.
\end{equation}
(b) $0\leq \theta<3-\sqrt{3}$, \eqref{lambda}, and
\begin{equation}\label{sec-mom}
\sum_{\lambda_n\leq T}\lambda_n^2|r_n|^2\ll T^\theta.
\end{equation}
Then $\phi(y)$ is a $B^2$-almost periodic function and therefore possesses a limiting distribution.
\end{cor}
Part (a) of this corollary is useful to apply when the $r_n$'s satisfy the nice bound $r_n \ll \lambda_n^{-\beta}$ where $\beta >1/2$. 
The existence of limiting distributions for \eqref{eq:Wild} and \eqref{eq:RSld} may be deduced from this case.  
If we assumed RH and  $|\zeta'(\rho)|^{-1} \ll |\rho|^{\frac{1}{2}-\varepsilon}$, then part (a) implies that 
\eqref{eq:Ngld} possesses a limiting distribution.  

On the other hand, if the $r_n$'s oscillate significantly, then by part (b) of the above corollary it suffices to have a modest
bound for the second moment of $\lambda_n |r_n|$.  

More generally, we prove a version of Theorem \ref{lim-dis} for vector-valued functions whose components are of the type $\phi(y)$. 
For instance, let $\vec{\phi}: [0,\infty) \to \mathbb{R}^\ell$ be given by 
\begin{equation}
   \label{eq:vecphi} 
   \vec{\phi}(y) = \big( \phi_1(y), \ldots, \phi_{\ell}(y)\big),
\end{equation}
where each $\phi_{k}(y)$ is of the shape \eqref{eq:phi}. Then we have the following.
\begin{theorem}\label{117}
Suppose that the conditions of Theorem \ref{lim-dis} or Corollary \ref{11} hold for each $\phi_k(y)$ for $1 \le k \le \ell$. Then $\vec{\phi}(y)$ possesses a limiting distribution.
\end{theorem}
\noindent This theorem contains as special cases the results of Wintner, Rubinstein and Sarnak, and Ng. That is, the functions in equations \eqref{eq:Wild}, \eqref{119}, \eqref{eq:RSld},
and \eqref{eq:Ngld} possess limiting distributions.

We also provide several new examples of functions which have limiting distributions. 
These functions are now described.

Let $\piup$  be an irreducible unitary cuspidal automorphic representation of ${\rm GL}_d(\mathbb{A}_{\mathbb{Q}})$, and let $L(s, \piup)$ be the automorphic $L$-function attached to $\piup$. We have
$$L(s, \piup)=\prod_{p < \infty} L_p(s, \piup_p),$$
where $$L_p(s, \piup_p)=\prod_{j=1}^{d} \left(1-\frac{\alpha_\piup (p, j)}{p^s} \right)^{-1},$$
for $\Re(s)>1$.  The completed $L$-function $\Phi(s, \piup)$ is defined by
$$\Phi(s, \piup)=L_\infty(s, \piup_\infty) L(s, \piup),$$
where  the Archimedean local factor is
$$L_\infty(s, \piup_\infty)=\prod_{j=1}^{d} \Gamma_\mathbb{R} (s+\mu_\piup(j))$$
and $\Gamma_\mathbb{R} (s)=\pi^{-s/2} \Gamma(s/2)$ where $\Gamma$ is the classical gamma function.
For $1\leq j\leq d$, the complex numbers $\alpha_\piup (p, j)$ and $\mu_\piup(j)$ are called the \emph{local parameters}. It is known that $\Phi(s, \piup)$ is entire (except in the case $L(s, \piup)=\zeta(s-i\tau_0)$ for $\tau_0\in \mathbb{R}$, which in this case $\Phi(s, \piup)$ has two simple poles) and satisfies the functional equation
$$\Phi(s, \piup)=\epsilon(s, \piup) \Phi(1-s, \tilde{\piup}),$$
with 
$$\epsilon(s, \piup)=\epsilon_\piup Q_\piup^{1/2-s},$$
where $Q_\piup\geq 1$ is an integer called the conductor of $\piup$, $\epsilon_\piup$ is the root number satisfying $|\epsilon_\piup|=1$, and $\tilde{\piup}$ is the representation contragredient to $\piup$. 
It is expected that all non-trivial zeros of $L(s, \piup)$ are located on the line $\Re(s)=1/2$ and this is 
known as the generalized Riemann hypothesis for $L(s, \piup)$.

We now consider prime counting functions associated to $L(s, \pi)$. 
Let 
\begin{equation}
\label{a_pi}
a_\piup(p^k)=\sum_{j=1}^{d} \alpha_\piup(p, j)^k,
\end{equation}
and define $$\psi(x, \piup)=\sum_{n\leq x} \Lambda(n) a_\piup(n),$$
where $\Lambda(n)$ is the classical von Mangoldt function. We have, for $\Re(s)>1$, 
$$-\frac{L^\prime(s, \piup)}{L(s, \piup)}=\sum_{n=1}^{\infty} \frac{\Lambda(n) a_\piup(n)}{n^s}.$$
The prime number theorem for $L(s, \piup)$ (see \cite[Theorem 2.3]{LY2}) is the assertion that
$$\psi(x, \piup)=\delta(x, \piup)+O(x\exp(-c\sqrt{\log{x}}))$$
for some positive constant $c$, where 
$$\delta(x, \piup)= \left\{  \begin{array}{ll} \frac{x^{1+i\tau_0}}{1+i\tau_0}&{\rm if}~L(s, \piup)=\zeta(s-i\tau_0),\\ 0&{\rm otherwise}.\end{array} \right.$$
From Corollary \ref{11} (a) we are able to deduce that a scaled version of the above error term
possesses a limiting distribution.

\begin{cor}
\label{339}
Under the assumption of the generalized Riemann hypothesis for $L(s, \piup)$ the function
$$E_1(y, \piup) = e^{-y/2}\left(\psi(e^y, \piup)-\delta(e^y, \piup)\right)$$
has a limiting distribution.
\end{cor}

Note that Wintner's theorem \eqref{eq:Wild} is a special case of the above corollary. 
In addition, for a modular newform $f$ of weight $k$ and level $N$, 
we conclude, under the assumption of the generalized Riemann hypothesis, that $e^{-y/2}\psi(e^y,f)$ has a limiting distribution.

We now introduce several other functions that possess limiting distributions.  These functions are related to certain negative moments 
of the derivative of  an $L$-function evaluated at its zeros. The first case to consider is the Riemann zeta function.
Gonek  \cite{Gonek} and Hejhal \cite{Hejhal} studied $J_{-1}(T) = \sum_{0<\gamma\leq T} |\zeta'(\rho)|^{-2}$ and Gonek 
conjectured that  
\[
    J_{-1}(T) \sim \frac{3}{\pi^3} T.
\]  
Assuming the Riemann hypothesis and all zeros of $\zeta(s)$ are simple, 
Milinovich and Ng \cite{MN} proved that
$J_{-1}(T)\geq(\frac{3}{2 \pi^3}-\varepsilon)T$ for every $\varepsilon >0$ and $T$ sufficiently large.
In our work,  we make the weaker assumption
\begin{equation} \label{330}
   J_{-1}(T) \ll T^{\theta} \text{ with } 1 \le \theta <  3-\sqrt{3}.  
\end{equation}
Currently,  assuming the Riemann hypothesis
and the simplicity of zeros of $\zeta(s)$, no upper bounds are known for $J_{-1}(T)$. However, the weak Mertens conjecture, the assumption
that
$$\int_1^X\Big( \frac{M(x)}{x} \Big)^2dx\ll\log X,$$
implies   
$|\zeta'(\rho)|^{-1}\ll|\rho|$ and thus $J_{-1}(T)\ll T^{3+\varepsilon}$ (see \cite[p. 377, eq. (14.29.4)]{T}).

We also require a version of \eqref{330} for Dirichlet $L$-functions $L(s,\chi)$.
We assume there exists a positive $\theta$
such that
\begin{equation}\label{135}
 \sum_{\chi \text{ mod } q}
\sum_{\substack{ 0 < |\Im(\rho_{\chi})| \le T \\ L(\rho_{\chi},\chi)=0}} |L'(\rho_{\chi},\chi)|^{-2} \ll_{q} T^{\theta} \text{ where } 1 \leq\theta <  3 -\sqrt{3}.
\end{equation}
It seems plausible that such a bound holds and it is natural to conjecture there is a positive constant $C_{q}$ such that
$$
\sum_{\chi \text{ mod } q} 
\sum_{ \substack{ 0 < |\Im(\rho_{\chi})| \le T \\ L(\rho_{\chi},\chi)=0}} |L'(\rho_{\chi},\chi)|^{-2}\sim C_{q} T.$$  
In fact, we can prove that this sum is greater than a positive constant times $T$,
assuming that all zeros of $L(s,\chi)$ are  simple and lie on the  critical line.  Finally, observe that \eqref{330} implies that 
all zeros of $\zeta(s)$ are simple and \eqref{135} implies that all nonreal zeros of the $L(s,\chi)$ are simple.   
We make use of these facts in our applications.

We shall introduce several other summatory functions. For $\alpha\in[0,1]$ and $x>0$, we set 
\begin{equation*}\label{114}
M_\alpha(x)=\sum_{n\leq x}\dfrac{\mu(n)}{n^\alpha}.
\end{equation*}
Over the years, there has been significant interest in these functions.
For instance, Landau showed in his Ph.D. thesis that $M_1(x)$ converges to 0. 
In 1897 Mertens conjectured that $M_0(x)=M(x)$ is bounded in absolute value by $\sqrt{x}$.  
This conjecture implies the Riemann hypothesis. 
Many researchers studied the size of $M_0(x)$.   Finally, in 1985, Odlyzko and te Riele \cite{Odlyzko} showed that Mertens' conjecture
is false.  On the Riemann hypothesis, it is known that $M(x) \ll x^{\frac{1}{2}+\varepsilon}$ for any $\varepsilon >0$.  Hence, by partial
summation, it follows that $M_{\alpha}(x)$ converges to $\zeta(\alpha)^{-1}$ for $\alpha > \frac{1}{2}$.  
Consequently, we observe that the behaviour of $M_{\alpha}(x)$ changes at $\alpha =1/2$
and thus define
\begin{equation}\label{115}
E_2(y,\alpha)=\left\{\begin{array}{lll}
e^{y(-1/2+\alpha)}M_\alpha(e^y)&\quad\text{if}&\quad0\leq\alpha\leq1/2,\\
e^{y(-1/2+\alpha)}\big(M_\alpha(e^y)- \frac{1}{\zeta(\alpha)} \big)&\quad\text{if}&\quad1/2<\alpha \le 1.
\end{array}\right.
\end{equation}

We now consider weighted sums of the Liouville function.  The Liouville function is given by $\lambda(n)=(-1)^{\Omega(n)}$
where $\Omega(n)$ is the total number of prime factors of $n$.   
We set
\begin{equation*}\label{116}
L_\alpha(x)=\sum_{n\leq x}\dfrac{\lambda(n)}{n^\alpha}.
\end{equation*}
P\'{o}lya and Tur\'{a}n studied $L_0(x)=L(x)$ and $L_1(x)$, respectively. 
Early numerical calculations suggested  that the inequalities  $L_0(x)\leq0$ and $L_1(x)>0$ hold  for all $x \ge 2$.
In 1958, Haselgrove 
\cite{Haselgrove} showed that $L_0(x)$ and $L_1(x)$ change sign infinitely often. 
Tanaka \cite{Tanaka} showed that the first value of $n$ 
for which $L_0(n)>0$ is 906,105,257. Borwein, Ferguson, and Mossinghoff \cite{Borwein} determined that the smallest 
value of $n$ for which $L_1(n)<0$ is 72,185,376,951,205.
It would be interesting to know how often $L_{\alpha}(x)$ is positive or negative. 
In order to study such questions we  define the error terms
\begin{equation}\label{58}
E_3(y,\alpha)=\left\{\begin{array}{lll}
e^{y(-1/2+\alpha)}L_\alpha(e^y)&\quad\text{if}&\quad0\leq\alpha<1/2,\\
e^{y(-1/2+\alpha)}\big(L_\alpha(e^y)- \frac{y}{2\zeta(1/2)} \big)&\quad\text{if}&\quad\alpha=1/2,\\
e^{y(-1/2+\alpha)}\big(L_\alpha(e^y)- \frac{\zeta(2\alpha)}{\zeta(\alpha)} \big)&\quad\text{if}&\quad1/2<\alpha\leq1.
\end{array}\right.
\end{equation}
In \cite{N} it was mentioned that $E_3(y,0)$ possesses a limiting distribution under the same hypotheses 
for which $e^{-y/2} M(e^y)$ possesses a limiting distribution.   Recently, Humphries \cite{Humphries}
studied these functions in the range $\alpha \in [0,1/2)$ and showed that, for these $\alpha$, the Riemann hypothesis and $J_{-1}(T) \ll T$ imply that $E_3(y,\alpha)$ possesses a limiting distribution. 

Our next example concerns the M\"{o}bius function in arithmetic progressions. 
For $q \ge 2$ and $(a,q)=1$, let 
$$M(x;q,a)=\sum_{\substack{n\leq x\\n\equiv a\ (\mathrm{mod}\ q)}}\mu(n).$$
This is a variant of $M(x)$ with the extra condition $n \equiv a\ (\mathrm{mod}\ q)$ inserted.  
Sums like $M(x;q,a)$ reflect the behaviour of 
primes in arithmetic progressions.  In fact, many theorems which can be established for 
\begin{equation*}
  \label{psixqa}
\sum_{\substack{n\leq x\\n\equiv a\ (\mathrm{mod}\ q)}}\Lambda(n)-\frac{x}{\phi(q)}
\end{equation*}
 have corresponding analogues for $M(x;q,a)$. 
For a fixed integer $q\geq2$, we define 
\begin{equation}\label{130}
   E_4(y;q,a) = e^{-y/2}  M(e^y;q,a).
\end{equation}
The next corollary establishes the existence of limiting distributions for $E_2(y,\alpha)$, $E_3(y,\alpha)$, and $E_4(y;q,a)$. 
\begin{cor}\label{132}  Let $\alpha \in [0,1]$, $q \ge 2$, and $(a,q)=1$.  \\
(i) If RH is true and \eqref{330} holds, then $E_2(y,\alpha)$ possesses a limiting distribution.\\
(ii) If RH is true and \eqref{330} holds, then $E_3(y,\alpha)$ possesses a limiting distribution.\\
(iii) If the generalized Riemann hypothesis is true for all Dirichlet $L$-functions modulo $q$ and \eqref{135} holds, then $E_4(y;q,a)$ possesses a limiting distribution.
\end{cor}
Part (i) improves and generalizes the main result of \cite{N}. Similarly, part (ii) improves and generalizes the limiting distribution result of \cite{Humphries}.
In \cite{N} and \cite{Humphries} the bound $J_{-1}(T) \ll T$ is employed, whereas we use the weaker bound \eqref{330}. 
It it possible that parts (i) and (ii) may be extended to hold for all $\alpha \in \mathbb{R}$.  In addition, assuming the same conditions as in part (iii), 
we can show 
that for $q \ge 2$ and $(a,q)=1$ that
$e^{-y/2} L(e^y;q,a)$ possesses a limiting distribution where 
$$L(x;q,a) = \sum_{\substack{m\leq x\\n\equiv a\ (\mathrm{mod}\ q)}}\lambda(n).$$

Our final example of error terms which possess limiting distributions is related to number fields. Let $K/k$ be a normal extension of number fields with Galois group $G = \mathrm{Gal}(K/k)$. Denote by $\mathcal{O}_{k}$ and $\mathcal{O}_{K}$ the corresponding rings of integers of $k$ and $K$. 
We define several counting functions.  Let
\[
   \pi_{k}(x) = \sum_{N\mathfrak{p} \le x} 1
\]
where $N \mathfrak{p}$ denotes the norm of the prime ideal $\mathfrak{p}  \subset \mathcal{O}_k$
and for a conjugacy class $C$ of $G$
\[
   \pi_{C}(x) = \sum_{\substack{N\mathfrak{p} \le x \\ \sigma_{\mathfrak{p}}=C}} 1
\]
where $\sigma_{\mathfrak{p}}$ is the Frobenius conjugacy class associated to $\mathfrak{p}$. 

Associated to $r$ distinct conjugacy classes $C_1, \ldots, C_r$ in $G$, we define
$$\vec{E}_5(y)=ye^{-y/2}\left( \frac{|G|}{|C_1|}\pi_{C_1}(e^y) - \pi_{k}(e^y) , \ldots , \frac{|G|}{|C_r|}\pi_{C_r}(e^y) - \pi_{k}(e^y) \right).$$ 
In order to study $\vec{E}_5(y)$, we require information regarding the zeros of  Artin $L$-functions associated to the extension $K/k$. 
Let $\rho$ be a representation of $G$ in $\mathrm{GL}_n(\mathbb{C})$ with character $\chi=\mathrm{tr}(\rho)$.
The principal character $\chi_0$ is the character attached to the trivial representation $\rho_0=1$. 
For each character $\chi$ of $G$, we associate  the Artin $L$-function $L(s,\chi,K/k)$. It is known that 
$L(s,\chi,K/k)$ is a meromorphic function on the complex plane. Moreover, there is the following fundamental conjecture.
\begin{conj}[Artin's Holomorphy Conjecture]\label{335}

If $\chi$ is non-trivial then $L(s,\chi,K/k)$ is entire.
\end{conj}
Also it is conjectured that an analogue of the Riemann hypothesis holds for  Artin $L$-functions. 

For further information regarding Artin $L$-functions see \cite[pp. 218--225]{Cassels} .

In his Ph.D. thesis \cite{NgPhd}, the second author showed that $\vec{E}_5(y)$ possesses a limiting distribution.
This can be deduced as a corollary of Theorem \ref{117}.   
\begin{cor}\label{118}
Under the assumptions of the generalized Riemann hypothesis and Artin's holomorphy conjecture for $L(s,\chi,K/k)$, where $\chi$ ranges through the irreducible characters of $G$, $\vec{E}_5(y)$ possesses a limiting distribution.
\end{cor}
\noindent This result contains as special cases the fact that \eqref{119} and \eqref{eq:RSld} possess limiting distributions.

The above corollaries are just a few applications of Theorems \ref{lim-dis} and \ref{117} and there are other interesting examples. 
For instance, Fiorilli \cite{F1}  applies our theorems in his work on highly biased prime number races
and also in his work \cite{F2} on prime number races associated to elliptic curves.

Our next theorem states that under an additional assumption on the exponent set $(\lambda_n)_{n \in \mathbb{N}}$ 
the Fourier transform of the limiting distribution of Theorem \ref{117} can be explicitly calculated. 
In order to explain our result we require some notation.

For $1\leq k\leq\ell$, let the component function $\phi_k(x)$ of \eqref{eq:vecphi} be defined by
$$\phi_k(x)=c_k+\Re\Big(\sum_{\lambda_{k,n}\leq X}r_{k,n}e^{i\lambda_{k,n}y}\Big)+\mathcal{E}_k(y,X),$$
where $c_k\in\mathbb{R}$, $(\lambda_{k,n})_{n \in \mathbb{N}} \subset \mathbb{R}^{+}$ is an increasing sequence, 
 $(r_{k,n})_{n\in \mathbb{N}} \subset \mathbb{C}$, and $\mathcal{E}_k(y,X)$ satisfies \eqref{9}.
Note that the collection of $(\lambda_{k,n})_{n \in \mathbb{N}}$ for $1 \le k \le \ell$ is a multiset.
We now consider the set $\cup_{k=1}^{\ell} \cup_{n=1}^{\infty} \{ \lambda_{k,n} \}$ and reorder its 
elements to construct the increasing sequence
$(\lambda_m)_{m \in \mathbb{N}}$. Also, we define
\begin{equation*}\label{338}
r_k(\lambda_m)=\left\{\begin{array}{ll}r_{k,n}&\mbox{if $\lambda_m=\lambda_{k,n}$ 
for some $n\in\mathbb{N}$,}\\0&\text{otherwise}.\end{array}\right.
\end{equation*}
With this notation in hand, we now provide a formula for the Fourier transform of the limiting distribution 
of $\vec{\phi}(y)$. 
\begin{theorem}\label{66}
Assume that $\mu$ is the limiting distribution associated to $\vec{\phi}(y)$ as given 
in Theorem \ref{117}. Suppose that the set $\{\lambda_m\}_{m \in \mathbb{N}}$ is linearly independent over $\mathbb{Q}$. 
Then the Fourier transform 
$$\hat{\mu}(\vec{\xi})=\int_{\mathbb{R}^\ell}e^{-i\sum_{j=1}^\ell\xi_jt_j}d\mu(t_1,\ldots,t_\ell)$$ 
of $\mu$ at $\vec{\xi}=(\xi_1,\ldots,\xi_\ell)\in\mathbb{R}^\ell$ 
exists and is equal to
$$\hat{\mu}(\vec{\xi})=\exp\big(-i{\textstyle\sum_{k=1}^\ell} c_k\xi_k\big)\times\prod_{m=1}^\infty J_0\big(\big|{\textstyle\sum_{k=1}^\ell} r_k(\lambda_m)\xi_k\big|\big),$$
where $J_0(z)$ is the Bessel function
\begin{equation*}
J_0(z)=\int_0^1e^{-iz\cos(2\pi t)}dt.
\end{equation*}
\end{theorem}

The above theorem is a useful tool in studying arithmetic  applications of our limiting distribution theorems. 
We now discuss an application.
For $q \ge 2$ and $a_1,\ldots,a_r$, $r$ distinct reduced residue classes mod $q$, consider the set
\[
  S_{q;a_1, \ldots, a_r} = \{x>0\mid M(x;q,a_1) > M(x;q,a_2) > \cdots > M(x;q,a_r)\}.
\]
In analogy to the Shanks-R\'enyi prime number race, we ask whether this set
contains infinitely many natural numbers and if it possesses a density. 
In this situation it is convenient to consider logarithmic density. 
\begin{definition}\label{503}
For $P\subseteq[0,\infty)$, set
$$\overline{\delta}(P)=\limsup_{X\to\infty}\dfrac{1}{\log X}\int_{t\in P\cap[2,X]}\dfrac{dt}{t}$$
and
$$\underline{\delta}(P)=\liminf_{X\to\infty}\dfrac{1}{\log X}\int_{t\in P\cap[2,X]}\dfrac{dt}{t}.$$
If $\overline{\delta}(P)=\underline{\delta}(P)=\delta(P)$, we say that the \emph{logarithmic 
density} of $P$ is $\delta(P)$.
\end{definition}
In order to study $ S_{q;a_1, \ldots, a_r}$, we consider
$$\vec{E}_6(y)=e^{-y/2}\big(M(e^y;q,a_1),\ldots,M(e^y;q,a_r)\big).$$
Theorem \ref{117} implies $\vec{E}_6(y)$ has a limiting distribution $\mu_{q;a_1,\ldots,a_r}$
assuming the generalized Riemann hypothesis for Dirichlet $L$-functions modulo $q$ and \eqref{135}.
If it were known that $\mu_{q;a_1,\ldots,a_r}$ is an absolutely continuous measure, then it would 
follow that 
\begin{equation} \label{deltaform}
   \delta( S_{q;a_1, \ldots, a_r}) = \mu(\{ x \in \mathbb{R}^r \ | \ x_1 > x_2 > \cdots > x_r \}).
\end{equation}
In order to show that $\mu_{q;a_1,\ldots,a_r}$ is absolutely continuous, we require further information
on the imaginary parts of zeros of Dirichlet $L$-functions.
We now recall a folklore conjecture concerning the  diophantine nature of the imaginary parts.
\begin{conj}[Linear Independence Conjecture] \label{LI}
The multiset of the nonnegative imaginary parts of the nontrivial zeros of Dirichlet $L$-functions corresponding
to primitive characters is linearly independent over the rationals. 

\end{conj}
With this conjecture in hand, it follows from Formula \eqref{500} and Theorem \ref{66} that
\begin{equation}
 \label{muhatformula}
\hat{\mu}_{q;a_1,\ldots,a_r}(\xi_1,\ldots,\xi_r)=\prod_{\chi\ \mathrm{mod}\ q}\prod_{\gamma_\chi>0}
J_0\bigg(\dfrac{2\big|\sum_{j=1}^r\chi(a_j)\xi_j\big|}{\varphi(q) \big|\rho_\chi L'(\rho_\chi,\chi)\big|}\bigg).
\end{equation}
Following the arguments in \cite[Lemma 2.1]{FeMa} and \cite[Lemma 6.4]{Humphries} we can deduce 
from \eqref{muhatformula} that $\mu_{q;a_1,\ldots,a_r}$ possesses a density function and  is absolutely continuous.
Thus  on the generalized Riemann hypothesis, Conjecture \ref{LI}, and \eqref{135} it follows from 
\eqref{deltaform} that $\delta( S_{q;a_1, \ldots, a_r})$ exists. We can also employ
\eqref{muhatformula} to investigate symmetries of the density function of $\mu_{q;a_1,\ldots,a_r}$.
The proof of Proposition 3.1 of \cite{RS} yields the following. 

\begin{prop}\label{502}
Assume the generalized Riemann hypothesis, \eqref{135}, and the linear independence conjecture. 
Then the density function of $\mu_{q;a_1,\ldots,a_r}$ is symmetric in $(t_1,\ldots,t_\ell)$
if and only if either $r=2$ or $r=3$ and there is $\rho\neq1$ such that $\rho^3\equiv1$, 
$a_2\equiv a_1\rho$, and $a_3\equiv a_1\rho^2$ (modulo $q$). 
\end{prop}

As a consequence of the symmetry of the density function of $\mu_{q;a_1,\ldots,a_r}$ we obtain 
the next corollary.
\begin{cor}\label{504}
Assume the conditions of Proposition \ref{502}. If either $r=2$ or $r=3$ and there is $\rho\neq1$ 
such that $\rho^3\equiv1$, $a_2\equiv a_1\rho$, and $a_3\equiv a_1\rho^2$ (modulo $q$), then 
$$\delta\big(\{x>0\mid M(x;q,a_1) > M(x;q,a_2) > \cdots > M(x;q,a_r)\}\big)=\dfrac{1}{r!}.$$
\end{cor}
\noindent In particular, if $a_1$ and $a_2$ are distinct residues modulo $q$,  
$$\delta\big(\{x>0\mid M(x;q,a_1) > M(x;q,a_2)\}\big)=\delta\big(\{x>0\mid M(x;q,a_2) > M(x;q,a_1)\}\big)=\dfrac{1}{2}.$$
This shows that the race between the summatory functions of the M\"{o}bius function on two arithmetic progressions is unbiased. 

Our general limiting distribution theorems can be used in 
proposing and studying many new arithmetic problems. For example, let
\begin{equation*}\label{120}
\vec{E}_7(y)=\big(ye^{-y/2}(\pi(e^y)-\mathrm{Li}(e^y)),e^{-y/2}M(e^y)\big).
\end{equation*}
Then Corollaries \ref{118} and \ref{132}(i)

imply that if the Riemann hypothesis and 
\eqref{330} hold, then $\vec{E}_7(y)$ possesses a limiting distribution. If in addition, the linear independence conjecture for the 
zeros of $\zeta(s)$ is true,  then 
$$\Bigg\{x>0\ \bigg|\ \dfrac{\log x}{\sqrt{x}}\big(\pi(x)-\mathrm{Li}(x)\big)>\dfrac{M(x)}{\sqrt{x}}\Bigg\}$$
possesses a logarithmic density. It would be interesting to determine the value of this logarithmic density.
However, this requires further analysis of the constructed distribution. 

As mentioned before, our strategy in the proof of our general limiting distribution theorem will be to prove that $\phi(y)$ is a $B^2$-almost periodic function.  Since Besicovitch \cite[Section 2]{Besicovitch2} proved that $B^2$-almost periodic functions satisfy a Parseval type identity, we deduce the following result.
\begin{theorem}\label{123}
Suppose that the function $\phi(y)$ of \eqref{eq:phi} satisfies the conditions of Theorem \ref{lim-dis} or Corollary \ref{11}. Then we have
\begin{equation}\label{345}
\lim_{Y\to\infty}\dfrac{1}{Y}\int_0^Y \phi(y)^2dy=c^2+{\textstyle\frac{1}{2}}\sum_{n=1}^\infty|r_n|^2.
\end{equation}
\end{theorem}
In fact, it is possible to show following an argument of Fiorilli \cite[Lemmas 2.4, 2.5]{F1} that 
\[
\lim_{Y\to\infty}\dfrac{1}{Y}\int_0^Y \phi(y)^2dy = \int_{\mathbb{R}} t^2 d \mu(t) 
\]
where $\mu$ is the limiting distribution associated 
to $\phi$.  A similar argument would also establish that   
\[
\lim_{Y\to\infty}\dfrac{1}{Y}\int_0^Y \phi(y)dy = \int_{\mathbb{R}} t d \mu(t) =c.
\]

As a corollary, we deduce Cramer's result \cite{Cramer} and 
its analogues for the error term of an automorphic $L$-function,  $e^{-y/2}M(e^y)$, and $e^{-y/2}L(e^y)$.
\begin{cor}\label{331}
(i) Let $L(s, \piup)$ be an automorphic $L$-function. If the generalized Riemann hypothesis is true for $L(s, \piup)$, then 
$$\lim_{Y\to\infty}\dfrac{1}{Y}\int_0^Y\bigg(\dfrac{\psi(e^y, \piup)-\delta(e^y, \piup)}{e^{y/2}}\bigg)^2dy=
4\ (\mathrm{ord}_{s=1/2}L(s, \piup))^2+\sum_{\substack{\gamma>0\\L(1/2+i\gamma, \piup)=0}}\dfrac{2 m_{\gamma}^2}{\frac{1}{4}+\gamma^2},$$
where $m_{\gamma}$ denotes the multiplicity of the zero ${1}/{2}+i \gamma$. \\
(ii) If the Riemann hypothesis is true and \eqref{330} holds, then 
$$\lim_{Y\to\infty}\dfrac{1}{Y}\int_0^Y\bigg(\dfrac{M(e^y)}{e^{y/2}}\bigg)^2dy=
\sum_{\substack{\gamma>0 \\ \zeta(\frac{1}{2}+i \gamma)=0}}\dfrac{2}{|\rho\zeta'(\rho)|^2}.$$
(iii) If the Riemann hypothesis is true and \eqref{330} holds, then 
$$\lim_{Y\to\infty}\dfrac{1}{Y}\int_0^Y\bigg(\dfrac{L(e^y)}{e^{y/2}}\bigg)^2dy=
\dfrac{1}{\zeta(\frac{1}{2})^2}+\sum_{\substack{\gamma>0 \\ \zeta(\frac{1}{2}+i \gamma)=0}}2\ \bigg|\dfrac{\zeta(2\rho)}{\rho\zeta'(\rho)}\bigg|^2.$$
\end{cor}
Note that Theorem \ref{331} (ii) improves Theorem 3 of \cite{N} where the stronger condition $J_{-1}(T) \ll T$ is assumed. 

The rest of this article is organized as follows. In Section 2 we review the background on $B^p$-almost periodic functions and show that almost periodic functions possess limiting distributions. In Section 3, we prove Theorem \ref{lim-dis} and Corollary \ref{11}.  
In Section 4, we deduce Corollaries \ref{339}, \ref{132}, and \ref{118}. In Section 5, we prove Theorem \ref{66}.  
Finally, we mention some notation used throughout this article. 
We write $f(x)=O(g(x))$ or $f(x) \ll g(x)$  to mean
there exists $M >0$ such that $|f(x)| \le M |g(x)|$ for all sufficiently large $x$.

\section{$B^p$-almost periodic functions and limiting distributions}

The main goal in this section is to provide the necessary background on  $B^p$-almost periodic functions needed in the proof of Theorem \ref{117}. It has been known since the 1930's
that any $B^p$-almost periodic function $\phi$\footnote{In this section $\phi$ denotes a Lebesgue integrable function.} possess limiting distributions.
Such a result is mentioned in \cite[Theorems 25 and 27]{JW} and proven in  
\cite[Theorem 8.3]{Wi2a}. However, the authors were unable to find a refereed publication 
from the 1930's which proves this result. The earliest journal publication we are aware of
is \cite{Bleher}, though it only proves the result for $\ell=1$. 
In order to keep our article self-contained, we provide a proof in the general case 
of a vector-valued function. 

We review some facts from the theory of almost periodic functions. 
Let $L_{loc}^p ([0, \infty))$ be the set of locally $p$-integrable functions on $[0, \infty)$. 
For $p\geq1$ and $\phi\in L_{loc}^p ([0, \infty))$, define 
$$\lVert\phi\rVert_{B^p}=\Big(\limsup_{Y\to\infty}\dfrac{1}{Y}\int_0^Y|\phi(y)|^pdy\Big)^{1/p}.$$
Denote by $\mathscr{T}$ the class of all real-valued trigonometric polynomials 
$$P_N(y)=\sum_{n=1}^Nr_ne^{i\lambda_ny}\quad (y\in\mathbb{R}),$$
where $r_n\in\mathbb{C}$ and $\lambda_n\in\mathbb{R}$.
The \emph{$B^p$-closure} of $\mathscr{S}$, denoted $\mathcal{H}_{B^p}(\mathscr{S})$, is the set of functions $\phi\in\mathscr{R}$ that satisfy the following property:\\
For any $\varepsilon>0$ there is a function $f_\varepsilon(y)\in\mathscr{S}$ such that
$$\lVert\phi(y)-f_\varepsilon(y)\rVert_{B^p}<\varepsilon.$$

\begin{definition}\label{1}
Any $\phi\in\cup_{p\geq 1}\mathcal{H}_{B^p}(\mathscr{T})$ is called an \emph{almost periodic function}. If $\phi\in\mathcal{H}_{B^p}(\mathscr{T})$ we say that $\phi$ is a $B^p$-\emph{almost periodic function}.
\end{definition}
\noindent For $\phi\in\mathcal{H}_{B^p}(\mathscr{T})$ and given $\varepsilon>0$ there exists 
\begin{equation}\label{50}
P_{N(\varepsilon)}(y)=\sum_{n=1}^{N(\varepsilon)}r_n(\varepsilon)e^{i\lambda_n(\varepsilon)y}
\end{equation}
in $\mathscr{T}$ such that 
\begin{equation*}
\lVert\phi(y)-P_{N(\varepsilon)}(y)\rVert_{B^p}<\varepsilon.
\end{equation*}
It is an important fact of the theory of almost periodic functions that in \eqref{50}, $\lambda_n(\varepsilon)$ can be taken only from a set $\Lambda(\phi)=\{\lambda_n\mid n\in\mathbb{N}\}$ and the corresponding values for $r_n$ are given by
$$r_n=\lim_{Y\to\infty}\dfrac{1}{Y}\int_0^Y\phi(y)e^{-i\lambda_ny}dy$$
(see \cite{Bleher}).

\begin{definition}\label{2}
A vector-valued function $\vec{\phi}:[0,\infty)\to\mathbb{R}^\ell$, $\vec{\phi}=(\phi_1,\ldots,\phi_\ell)$, is called \emph{almost periodic}, if there is a $p\geq 1$ such that each  component function $\phi_k$ ($1\leq k \leq \ell$) belongs to $\mathcal{H}_{B^p}(\mathscr{T})$. Moreover $\vec{\phi}$ is called $B^p$-\emph{almost periodic} if each $\phi_k$ ($1\leq k \leq \ell$) is $B^p$-{almost periodic}.
\end{definition}
\noindent It is known that $\mathcal{H}_{B^p}(\mathscr{T})\subseteq\mathcal{H}_{B^q}(\mathscr{T})$ if $1\leq q\leq p$ (see \cite[p. 476]{Bleher}). So a vector-valued function is almost periodic if and only if each of its component functions is almost periodic.\\

The following lemma states a version of Kronecker-Weyl equidistribution theorem.
\begin{lemma}\label{34}
Let $t_1,\dots,t_N$ be arbitrary real numbers. Suppose that $A$ is the topological closure of $\left\{y(t_1,\ldots,t_N)\mid y\in\mathbb{R}\right\}/\mathbb{Z}^N$ 
in the torus $\mathbb{T}^{N}$. Let $g:\mathbb{R}^{N}\rightarrow\mathbb{R}$ be a continuous function of period 1 in each of its variables. Then we have
\begin{equation*}\label{35}
\lim_{Y\rightarrow\infty}\dfrac{1}{Y}\int_0^Yg(yt_1,\ldots,yt_N)dy=\int_Ag(a)d\omega
\end{equation*}
where $\omega$ is the normalized Haar measure on $A$.
\end{lemma}
\begin{proof}
This may be deduced from the Kronecker-Weyl theorem (see \cite[pp. 1--16]{Hlawka}), and is also a special case
of Ratner's theorem on unipotent flows (see \cite{Morris}). 
\end{proof}
\noindent Next we prove that every vector-valued function whose components are real-valued trigonometric polynomials 
has a limiting distribution.
\begin{prop}\label{36}
For $1\leq k\leq\ell$, let $(\lambda_{k,n})_{n=1}^{N_k}$ be a real sequence and $(r_{k,n})_{n=1}^{N_k}$ be a complex sequence. Set
$$P_k(y) = \sum_{n=1}^{N_k}r_{k,n}e^{i\lambda_{k,n}y}\quad(y\in\mathbb{R}).$$
If $P_k(y)\in\mathbb{R}$ for all $y\in \mathbb{R}$, then
$$\vec{P}(y)=\big(P_1(y),\ldots,P_\ell(y)\big)$$
has a limiting distribution.
\end{prop}
\begin{proof}
We consider the set $\cup_{k=1}^{\ell} \cup_{n=1}^{N_k} \{ \lambda_{k,n} \}$ and write its elements in increasing order as the sequence $(\lambda_m)_{m=1}^N$.  For $1\leq k\leq\ell$, we set 
\begin{equation*}
r_k(\lambda_m)=\sum_{\substack{{1\leq n \leq N_k}\\{\lambda_{k, n}=\lambda_m}}} r_{k,n}.
\end{equation*}
Let $f:\mathbb{R}^\ell\to\mathbb{R}$ be a bounded continuous function. Suppose that $X:\mathbb{T}^{N}\rightarrow\mathbb{R}^\ell$ and $g:\mathbb{T}^{N}\rightarrow\mathbb{R}$ are defined by
$$X(\theta_1,\dots,\theta_N)=\bigg(\sum_{m=1}^Nr_1(\lambda_m)e^{2\pi i\theta_m},\ldots,\sum_{m=1}^Nr_\ell(\lambda_m)e^{2\pi i\theta_m}\bigg)$$
and $g(\theta_1,\dots,\theta_N)=f\big(X(\theta_1,\dots,\theta_N)\big)$. 
By applying Lemma \ref{34} with $t_1=\frac{\lambda_1}{2\pi},\ldots,t_N=\frac{\lambda_N}{2\pi}$, we have
$$\lim_{Y\rightarrow\infty}\dfrac{1}{Y}\int_0^Yg\big({\textstyle\frac{y\lambda_1}{2\pi},\ldots,\frac{y\lambda_N}{2\pi}}\big)dy=\int_Ag(a)d\omega,$$
where $A$ is the closure of $\big\{y\big(\frac{\lambda_1}{2\pi},\dots,\frac{\lambda_N}{2\pi}\big)
\mid y\in\mathbb{R}\big\}/\mathbb{Z}^N$ in $\mathbb{T}^{N}$ and $\omega$ is the normalized Haar measure on $A$. 
Define a probability measure $\mu_N$ on $\mathbb{R}^\ell$ by $\mu_N(B)=\omega\big(X^{-1}(B)\cap A\big)$, 
where $B$ is any Borel set in $\mathbb{R}^\ell$. By the change of variable formula 
\cite[Theorem 16.12]{Billingsley},  
\begin{equation}\label{348}
\int_Ag(a)d\omega=\int_{\mathbb{R}^\ell}fd\mu_N
\end{equation}
and thus
\begin{equation*}\label{349}
\lim_{Y\rightarrow\infty}\dfrac{1}{Y}\int_0^Yf\big(\vec{P}(y)\big)dy=\int_{\mathbb{R}^\ell}fd\mu_N,
\end{equation*}
for all bounded continuous real-valued functions $f$ on $\mathbb{R}^\ell$. Therefore, $\vec{P}(y)$ has a limiting distribution.
\end{proof}
Our next goal is to show that every almost periodic function possesses a limiting distribution. 
This requires several concepts from probability. 
\begin{definition}\label{401} 
Let $( \mu_n)_{n  \in \mathbb{N}}$ be a sequence of finite measures on a measurable space $X$. We say that $\mu_n$
\emph{converges weakly} to $\mu$ if for every bounded real-valued continuous function $f$ we have 
\begin{equation}
  \label{weakconv}
\int_X fd\mu_n\to\int_X fd\mu
\end{equation}
as $n\rightarrow \infty$.
\end{definition}
In fact, it is well known that \eqref{weakconv} only needs to be verified for Lipschitz functions. 
\begin{lemma}[Portmanteau]\label{402}
$\mu_n$ converges weakly to $\mu$ if and only if
$$\int_X fd\mu_n\to\int_X fd\mu$$
for any bounded Lipschitz function $f$ on $X$.
\end{lemma}
\begin{proof}
See \cite[Theorem 3.5]{Molch}.
\end{proof}
Next we define the tightness of a sequence of probability measures.
\begin{definition}\label{560}
A sequence $(\mu_n)_{n\in\mathbb{N}}$ of probability measures on $\mathbb{R}^\ell$ is \emph{tight} 
if for any $\varepsilon>0$ there is $A_\varepsilon>0$ such that $\int_{|\mathbf{x}|\geq A_\varepsilon}d\mu_n<\varepsilon$, 
for all $n\in\mathbb{N}$.
\end{definition}
The following lemma illustrates the importance of a tight sequence of measures.
\begin{lemma}[Helly's Selection Theorem]\label{561}
Let $(\mu_n)_{n\in\mathbb{N}}$ be a sequence of probability measures on $\mathbb{R}^\ell$. Then 
$(\mu_n)_{n\in\mathbb{N}}$ is tight if and only if for every subsequence $(\mu_{n_j})_{j \in \mathbb{N}}$ there is a further 
subsequence $(\mu_{n_{j_k}})_{k \in \mathbb{N}}$ and a probability measure $\mu$ such that $\mu_{n_{j_k}}$ 
converges weakly to $\mu$.
\end{lemma}
\begin{proof}
See \cite[Theorems 25.8 and 25.10]{Billingsley}.
\end{proof}
We are ready to prove the main result of this section.
\begin{theorem}\label{340}
Every almost periodic function possesses a limiting distribution.
\end{theorem}
\begin{proof}
Consider an almost periodic function $\vec{\phi}:[0,\infty)\to\mathbb{R}^\ell$. For $Y\geq1$, let
$$\nu_Y(B)=\dfrac{1}{Y}\mathrm{meas}\big([0,Y]\cap(\vec{\phi})^{-1}(B)\big)$$
for any Borel set $B$ in $\mathbb{R}^\ell$, where $\mathrm{meas}(\cdot)$ is the Lebesgue measure on $\mathbb{R}$. 

Note that, by Definition \ref{401}, $\vec{\phi}(y)$ has a limiting distribution if and only if there exists a probability measure $\mu$ such that the sequence $(\nu_Y)_{Y \in \mathbb{N}}$ 
converges weakly to $\mu$. By Lemma \ref{402} this is equivalent to 
$$\int_{\mathbb{R}^\ell} fd\nu_Y\to\int_{\mathbb{R}^\ell} fd\mu,$$
as $Y\to\infty$, for any bounded Lipschitz function $f:\mathbb{R}^\ell\to\mathbb{R}$. 

Now let $\vec{\phi}(y)=\big(\phi_1(y),\ldots\phi_\ell(y)\big)$ such that $\phi_k(y)$ belongs to $\mathcal{H}_{B^1}(\mathscr{T})$
for $1 \le k \le \ell$.  (Recall that $\mathcal{H}_{B^p}(\mathscr{T})\subseteq \mathcal{H}_{B^1}(\mathscr{T})$ for any $p\geq 1$.) Then for each component $\phi_k(y)$ and for $M\in\mathbb{N}$, there exists $N_k(M) \in \mathbb{N}$
and sequences $( r_{k,n})_{n=1}^{N_k}$  and $(\lambda_{k,n})_{n=1}^{N_k}$  such that 
\begin{equation}\label{350}
\limsup_{Y\to\infty}\dfrac{1}{Y}\int_0^Y\Big|\phi_k(y)-\sum_{n=1}^{N_k(M)}r_{k,n}e^{iy\lambda_{k,n}}\Big|dy<\dfrac{1}{M}.
\end{equation}
By Proposition \ref{36},
\begin{equation}
\label{Papp}
\vec{P}_M(y)=\bigg(\sum_{n=1}^{N_1(M)}r_{1,n}e^{iy\lambda_{1,n}},\ldots,\sum_{n=1}^{N_\ell(M)}r_{\ell,n}e^{iy\lambda_{\ell,n}}\bigg)
\end{equation}
has a limiting distribution $\mu_M$, i.e.
$$\lim_{Y\to\infty}\dfrac{1}{Y}\int_0^Yf\big(\vec{P}_M(y)\big)dy=\int_{\mathbb{R}^\ell}f(x) d\mu_M(x):=\mu_M(f),$$
for all bounded continuous functions $f:\mathbb{R}^\ell\to\mathbb{R}$. From now on for a probability measure $\nu$ on $\mathbb{R}^\ell$ and a function $g$,  we shall make use of the notation 
\[ 
  \nu(g) =\int_{\mathbb{R}^\ell}g(x) d\nu(x).
\]

Let $f:\mathbb{R}^\ell\to\mathbb{R}$ be a bounded Lipschitz function which satisfies 
\[
   |f(x)-f(y)| \le c_f |x-y|
\]
for all $x, y \in \mathbb{R}^\ell$ where $c_f$ is the Lipschitz constant. Then we have
\begin{eqnarray}\label{42}
\dfrac{1}{Y}\int_0^Yf\big(\vec{\phi}(y)\big)dy\leq\dfrac{1}{Y}\int_0^Yf\big(\vec{P}_M(y)\big)dy+\dfrac{c_f}{Y}\int_0^Y\big|\vec{\phi}(y)-\vec{P}_M(y)\big|dy
\end{eqnarray}
and
\begin{eqnarray}\label{43}
\dfrac{1}{Y}\int_0^Yf\big(\vec{\phi}(y)\big)dy\geq\dfrac{1}{Y}\int_0^Yf\big(\vec{P}_M(y)\big)dy-\dfrac{c_f}{Y}\int_{0}^Y\big|\vec{\phi}(y)-\vec{P}_M(y)\big|dy
\end{eqnarray}
for any $Y>0$ and $M\in \mathbb{N}$. Moreover,
$$\dfrac{1}{Y}\int_0^Y\big|\vec{\phi}(y)-\vec{P}_M(y)\big|dy\leq\sum_{k=1}^\ell\dfrac{1}{Y}\int_0^Y\Big|\phi_k(y)-\sum_{n=1}^{N_k(M)}r_{k,n}e^{iy\lambda_{k,n}}\Big|dy.$$
If we apply the latter inequality in \eqref{42} and \eqref{43} and take $\limsup$ and $\liminf$ as $Y\to\infty$, respectively, by employing \eqref{350} we obtain
\begin{equation}\label{37}
\mu_M(f)-\ell c_f/M\leq\liminf_{Y\to\infty}\dfrac{1}{Y}\int_0^Yf\big(\vec{\phi}(y)\big)dy\leq\limsup_{Y\to\infty}\dfrac{1}{Y}\int_0^Yf\big(\vec{\phi}(y)\big)dy\leq\mu_M(f)+\ell c_f/M.\\
\end{equation}
These inequalities imply that 
$L(f):=\lim_{Y\to\infty}\frac{1}{Y}\int_0^Yf\big(\vec{\phi}(y)\big)dy$ 
exists. Moreover, \eqref{37} implies that 
\begin{equation} \label{351}
\lim_{Y\to\infty}\nu_Y(f)=\lim_{M\to\infty}\mu_M(f)=L(f)
\end{equation}
exists for every bounded Lipschitz function $f:\mathbb{R}^\ell\to\mathbb{R}$.

We next show that $(\nu_Y)_{Y \in \mathbb{N}}$ is tight, i.e. for any $\varepsilon>0$, there is $A_\varepsilon>0$ such that 
\begin{equation*} \label{352}
\int_{|\mathbf{x}|\geq A_\varepsilon}d\nu_Y<\varepsilon
\end{equation*}
for all $Y\in \mathbb{N}$.
Let $\varepsilon>0$ be given. We choose a natural number $M$ such that ${\ell}/M<\varepsilon$. By \eqref{350} and \eqref{Papp}, there exists a vector function $\vec{P}_M(y)$ with trigonometric polynomials as its components such that
\begin{equation}
\label{difference}
|\vec{\phi}(y)-\vec{P}_M(y)|<{\ell}/M<\varepsilon,
\end{equation}
where $|.|$ denote the Euclidean norm in $\mathbb{R}^\ell$. Let $$A_\varepsilon=\sup_{y\in [0, \infty]} |\vec{P}_M(y)|+1.$$
Now by employing \eqref{difference} we have
$$\int_{|\mathbf{x}|\geq A_\varepsilon}d\nu_Y=\dfrac{1}{Y}\mathrm{meas}\big\{0\leq y \leq Y,~|\vec{\phi}(y)|>A_\varepsilon\big\}\leq \dfrac{1}{Y} \int_{0}^{Y} |\vec{\phi}(y)-\vec{P}_M(y)| dy<\varepsilon.$$
Hence $(\nu_Y)_{Y\in \mathbb{N}}$ is tight, as we stated.
Thus, by Lemma \ref{561}, 
there is a subsequence $(\nu_{Y_j})_{j \in \mathbb{N}}$ of $(\nu_Y)_{Y \in \mathbb{N}}$ and a probability measure $\mu$ on $\mathbb{R}^\ell$ 
such that 
\begin{equation*}
L(f)=\lim_{j\to\infty}\nu_{Y_j}(f)=\mu(f).
\end{equation*}
This together with \eqref{351} shows that
\begin{equation} \label{353}
  \lim_{Y\to\infty}\nu_{Y}(f)=\lim_{M\to\infty}\mu_{M}(f)=\mu(f),
\end{equation}
for every bounded Lipschitz function $f:\mathbb{R}^\ell\to\mathbb{R}$ and the proof is complete.
\end{proof}
\section{Proof of the main theorem}

The goal in this section is to prove Theorem \ref{117}. By Theorem \ref{340}, we know that $\vec{\phi}(y)$ has a limiting distribution if $\vec{\phi}(y)$ is a $B^2$-almost periodic function. Since $\vec{\phi}(y)$ is a $B^2$-almost periodic function if and only if each of its component functions $\phi_k(y)$ is $B^2$-almost periodic, Theorem \ref{117} will follow as a consequence of Theorem \ref{lim-dis}.

The proof of Theorem \ref{lim-dis} under conditions given in (a) uses a lemma of Gallagher. The proof under the assumptions given in (b) follows an argument first employed by Cram\'{e}r \cite{Cramer} and later used by Ng \cite{N}.

For a proof of the following lemma see \cite[Lemma 1]{Gal}.
\begin{lemma}[Gallagher]\label{gal}
Let $(\nu_n )_{n \in \mathbb{N}}$ be an arbitrary sequence of real numbers and $(c_n)_{n \in \mathbb{N}} \subset\mathbb{C}$. 
Assume that $f(x)=\sum_{n=1}^\infty c_ne^{2\pi i\nu_nx}$ is absolutely convergent. Then, for $U\geq0$,
$$\int_{-U}^U|f(x)|^2dx\ll\dfrac{1}{U^2}\int_{-\infty}^\infty\Big|\sum_{t<\nu_n\leq t+1/U}c_n\Big|^2dt.$$
\end{lemma}

\begin{proof}[Proof of Theorem \ref{lim-dis}]
Let $X>T>1$ and $V\geq0$.  Assume either of the conditions of Theorem \ref{lim-dis}.
We shall begin by showing that there exist $\delta\geq0$ and $\eta>0$ such that
\begin{equation}\label{alm-per}
\int_{V}^{V+1}\Big|\sum_{T<\lambda_n\leq X}r_ne^{iy\lambda_n}\Big|^2dy\ll 
\dfrac{(\log T)^\delta}{T^\eta}.
\end{equation}
First assume that Condition (a) in Theorem \ref{lim-dis} holds. Then
$$\int_V^{V+1}\Big|\sum_{T<\lambda_n\leq X}r_ne^{iy\lambda_n}\Big|^2dy\leq2\pi
\int_{-(V+1)}^{V+1}\Big|\sum_{T<\lambda_n\leq X}r_ne^{2\pi iy\lambda_n}\Big|^2dy.$$
Lemma \ref{gal} implies
\setlength\arraycolsep{0.1em}
\begin{eqnarray*}
\int_{-(V+1)}^{V+1}\Big|\sum_{T<\lambda_n\leq X}r_ne^{2\pi iy\lambda_n}\Big|^2dy
&&\ll\dfrac{1}{(V+1)^2}\int_{-\infty}^\infty\Big|\sum_{\substack{T<\lambda_n\leq X\\t<\lambda_n\leq t+\frac{1}{V+1}}}
r_ne^{2\pi it\lambda_n}\Big|^2dt\\&&\leq\int_{-\infty}^\infty\Big(
\sum_{\substack{T<\lambda_n\leq X\\t<\lambda_n\leq t+1}}|r_n|\Big)^2dt.
\end{eqnarray*}
In the last integral, $t$ satisfies $T-1\leq t\leq X$. From \eqref{SI} and $\beta>1/2$, we have
$$\int_{-\infty}^\infty\Big(\sum_{\substack{T<\lambda_n\leq X\\t<\lambda_n\leq t+1}}|r_n|\Big)^2dt
\leq\int_{T-1}^X\Big(\sum_{t<\lambda_n\leq t+1}|r_n|\Big)^2dt\ll\int_{T-1}^X
\dfrac{(\log t)^{2\gamma}}{t^{2\beta}}dt\ll\dfrac{(\log T)^{2\gamma}}{T^{2\beta-1}}.$$
So \eqref{alm-per} holds for $\delta=2\gamma$ and $\eta=2\beta-1$.

Next assume that (b) holds. Note that, by dyadic summations, \eqref{AI} and $\alpha \ge \beta$ imply
\begin{equation*}
\label{relation2}
\sum_{\lambda_n\leq T} |r_n|\ll T^{\alpha-\beta}(\log{T})^{\gamma+1}.
\end{equation*}
Thus, by partial summation, we conclude that if $\kappa>\alpha-\beta$ and $\nu>0$, then
\begin{equation}
\label{partial}
\sum_{\lambda_n>T} \frac{|r_n|(\log{\lambda_n})^{\nu}}{\lambda_n^\kappa}
\ll \frac{(\log{T})^{\gamma+\nu+1}}{T^{\kappa-\alpha+\beta}}.
\end{equation}
Since $|z|^2=z\bar{z}$, we have
\setlength\arraycolsep{0.1em}
\begin{eqnarray*}
\int_{V}^{V+1}\Big|\sum_{T<\lambda_n\leq X}r_ne^{iy\lambda_n}\Big|^2dy
&&=\sum_{T<\lambda_n\leq X}\sum_{T<\lambda_m\leq X}r_n\overline{r_m}\int_{V}^{V+1}e^{iy(\lambda_n-\lambda_m)}dy\\
&&\ll\sum_{T<\lambda_n\leq X}\sum_{T<\lambda_m\leq X}|r_nr_m|\min\left(1,\dfrac{1}{|\lambda_n-\lambda_m|}\right)
=\Sigma_1+\Sigma_2,
\end{eqnarray*}
where $\Sigma_1$ is the sum of those terms for which we have $|\lambda_n-\lambda_m|<1$, and $\Sigma_2$ 
is the sum of the rest of the terms. For $\Sigma_1$, by employing \eqref{AI} and \eqref{partial} we have
\setlength\arraycolsep{0.1em}
\begin{equation}
 \label{Sigma1bd}
\Sigma_1\leq\sum_{T<\lambda_n\leq X}|r_n|\sum_{\lambda_n-1 < \lambda_m < \lambda_n+1}|r_m|
\ll\sum_{\lambda_n>T}\frac{|r_n|(\log{\lambda_n})^{\gamma}}{\lambda_n^\beta}
\ll\frac{(\log{T})^{2\gamma+1}}{T^{2\beta-\alpha}}.
\end{equation}
Note that the last inequality is justified since \eqref{alphacond} implies that $\beta>\alpha-\beta$. 
To study $\Sigma_2$, we define for any $T\geq1$
$$S_T(U)=\sum_{\substack{\lambda_m> T\\|U-\lambda_m|\geq1}}\dfrac{|r_m|}{|U-\lambda_m|},$$
where $U\geq T$. Then we can write
\setlength\arraycolsep{0.1em}
\begin{eqnarray*}\label{19}
\Sigma_2=\sum_{T<\lambda_n\leq X}|r_n|\sum_{\substack{T<\lambda_m\leq X\\|\lambda_n-\lambda_m|\geq1}}
\dfrac{|r_m|}{|\lambda_n-\lambda_m|}\leq\sum_{T<\lambda_n\leq X}|r_n|S_T(\lambda_n).
\end{eqnarray*}
We determine an upper bound for $S_T(U)$ as follows. Let $0<c<1$ and $T\geq1$ be fixed. 
For any number $U\geq T$ consider the set of numbers $U^c$, $U-U^c$, and $U-1$. 
Either of the following cases occurs
\begin{equation*}\label{346}
T\leq U^c,\quad U^c<T\leq U-U^c,\quad U-U^c<T\leq U-1,\quad \mathrm{or}\quad U-1<T\leq U.
\end{equation*}
Suppose that the first case happens, i.e. $T\leq U^c$. Then
\setlength\arraycolsep{0.1em}
\begin{equation*}\label{20}
S_T(U)=\bigg(\sum_{T<\lambda_m\leq U^c}+\sum_{U^c < \lambda_m\leq U-U^c}+\sum_{U-U^c < \lambda_m\leq U-1}\nonumber
+\sum_{U+1\leq \lambda_m\leq U+U^c}+\sum_{U+U^c < \lambda_m\leq 2U}+\sum_{\lambda_m >  2U}\bigg)
\dfrac{|r_m|}{|U-\lambda_m|}.
\end{equation*}
Denote these six sums by $\sigma_1,\dots,\sigma_6$. Then, by applying \eqref{AI}, we deduce
\begin{eqnarray*}\label{21}
\sigma_1&&\leq\dfrac{1}{U-U^c}\sum_{T < \lambda_m \leq U^c}|r_m|
\ll  \frac{(U^c-T)^{\alpha} (\log U)^{\gamma}}{(U-U^c)T^{\beta}} \ll \frac{(\log{U})^\gamma}{U^{1-c\alpha}},
\end{eqnarray*}
\begin{eqnarray*}\label{22}
\sigma_2&&\leq\dfrac{1}{U^c}\sum_{U^c<\lambda_m\leq U-U^c}|r_m|\ll 
\frac{1}{U^c} \frac{(U-2U^c)^{\alpha} \log(U-U^c)^{\gamma}}{(U^c)^{\beta}}
\ll \frac{(\log{U})^\gamma}{U^{c+c\beta-\alpha}},
\end{eqnarray*}
\begin{eqnarray*}\label{23}
\sigma_3&&\leq\sum_{U-U^c < \lambda_m \leq U-1}|r_m|\ll
\frac{(U^c)^{\alpha}(\log U)^{\gamma}}{(U-U^c)^{\beta}}
\ll \frac{(\log{U})^\gamma}{U^{\beta-c\alpha}},
\end{eqnarray*}
\begin{eqnarray*}\label{24}
\sigma_4&&\leq\sum_{U+1 \leq \lambda_m \leq U+U^c}|r_m|\ll  
\frac{(U^c)^{\alpha}(\log U)^{\gamma}}{U^{\beta}} 
\ll \frac{(\log{U})^\gamma}{U^{\beta-c\alpha}},
\end{eqnarray*}
and
\begin{eqnarray*}\label{25}
\sigma_5&&\leq\dfrac{1}{U^c}\sum_{U+U^c < \lambda_m\leq 2U} |r_m| \ll
\frac{1}{U^c} \frac{U^{\alpha} (\log U)^{\gamma}}{U^{\beta}}
\ll \frac{(\log{U})^\gamma}{U^{c+\beta-\alpha}}.
\end{eqnarray*}
For $\sigma_6$, we divide the interval of summation into subintervals $(2^kU,2^{k+1}U]$ to get
\begin{eqnarray*}\label{26}
\sigma_6&&\leq\sum_{k=1}^\infty\dfrac{1}{(2^k-1)U}\sum_{2^kU < \lambda_m \leq 2^{k+1}U}|r_m|
\ll\left(\sum_{k=1}^\infty\dfrac{k^\gamma}{2^{k(\beta+1-\alpha)}}\right) \dfrac{(\log U)^\gamma}{U^{\beta+1-\alpha}}
\ll \frac{(\log{U})^\gamma}{U^{\beta+1-\alpha}},
\end{eqnarray*}
which is justified since $\alpha<\beta+1$ by \eqref{alphacond}. 
We observe that $\sigma_6 \ll \sigma_5 \ll \sigma_2$, $\sigma_4 \ll \sigma_3$, and $\sigma_1 \ll \sigma_3$ since $\beta \le 1$.
Thus we have 
\[
  S_T(U) \ll \sigma_2+\sigma_3 \ll \frac{(\log{U})^\gamma}{U^{c+c\beta-\alpha}}+
   \frac{(\log{U})^\gamma}{U^{\beta-c\alpha}}.
\]
In the last inequality we choose $c=\frac{\alpha+\beta}{\alpha+\beta+1}$
and hence  if $T\leq U^c$, then
\begin{equation*}
\label{S_n}
S_T(U)\ll \frac{(\log{U})^\gamma}{U^{\xi}},
\end{equation*}
where 
$$\xi=\frac{\beta^2-\alpha^2+\beta}{\alpha+\beta+1}.$$
By similar arguments, we find the same bound for $S_T(U)$ in the three other cases of \eqref{346}. 
Condition $\alpha^2+\alpha/2<\beta^2+\beta$ yields $\xi>\alpha-\beta$. Hence \eqref{partial} implies
\begin{equation}
\label{Sigma2bd}
\Sigma_2\ll\sum_{\lambda_n> T}|r_n|S_T(\lambda_n)\ll \frac{(\log{T})^{2\gamma+1}}{T^{\xi-\alpha+\beta}},
\end{equation}
where 
$$\xi-\alpha+\beta=\frac{2(\beta^2-\alpha^2)+(2\beta-\alpha)}{\alpha+\beta+1}.$$
By \eqref{Sigma1bd} and \eqref{Sigma2bd}, we have 
$$\Sigma_1+\Sigma_2 \ll \frac{(\log{T})^{2\gamma+1}}{T^{\xi-\alpha+\beta}},$$
since $\alpha,\beta >0$ implies that  $\xi-\alpha+\beta<2\beta-\alpha$.
Thus \eqref{alm-per} holds for $\delta=2\gamma+1$ and $\eta=\xi-\alpha+\beta$.

Now we show that \eqref{alm-per} together with \eqref{9} imply that $\phi(y)$ is a $B^2$-almost periodic function. 
It follows from \eqref{eq:phi} that for $e^{Y}>T \ge X_0$ and $y\geq y_0$, 
\begin{equation*} \label{27}
  \phi(y) - c - \Re \Big( \sum_{\lambda_n \le T}r_ne^{i\lambda_ny} \Big) 
   = \Re \Big( \sum_{T <\lambda_n \le e^{Y}}r_ne^{i\lambda_ny} \Big)+ \mathcal{E}(y,e^{Y}). 
\end{equation*}
Then, by employing \eqref{alm-per} and \eqref{9}, we obtain
\setlength\arraycolsep{0.1em}
\begin{eqnarray*}\label{28}
&& \limsup_{Y\to\infty}\dfrac{1}{Y}\int_{y_0}^Y\Big| \phi(y) - c - \Re \Big( \sum_{\lambda_n \le T}r_ne^{i\lambda_ny} \Big)  \Big|^2dy \\
&&\ll\limsup_{Y\to\infty}\dfrac{1}{Y}\int_{y_0}^Y\Big|\sum_{T<\lambda_n\leq e^Y}
r_ne^{iy\lambda_n}\Big|^2dy +\lim_{Y\to\infty}\dfrac{1}{Y}\int_{y_0}^Y|\mathcal{E}(y,e^Y)|^2dy\nonumber\\
&&\ll\limsup_{Y\to\infty}\dfrac{1}{Y}\sum_{j=0}^{\lfloor Y-y_0\rfloor}\int_{y_0+j}^{y_0+j+1}\Big|
\sum_{T<\lambda_n\leq e^Y}r_ne^{iy\lambda_n}\Big|^2dy\ll \dfrac{(\log T)^\delta}{T^\eta}.\nonumber\\
\end{eqnarray*}
This inequality together with 
$$\lim_{Y \to \infty} \frac{1}{Y} \int_{0}^{y_0}  \Big|\phi(y)-c-\Re \Big( \sum_{\lambda_n \le T} r_n e^{i \lambda_n y} \Big) \Big|^2 dy =0$$
imply that $\phi(y)$ is $B^2$-almost periodic. Hence, the theorem follows from 
Theorem \ref{340}.
\end{proof}
We next prove Corollary \ref{11}.
\begin{proof}[Proof of Corollary \ref{11}]
(a) Since $r_n \ll \lambda_n^{-\beta}$ then \eqref{lambda} imply that
$$\sum_{T<\lambda_n \leq T+1}|r_n|\ll\dfrac{\log T}{T^\beta}.$$
Now Theorem \ref{lim-dis}(a) implies the result.\\
(b) By partial summation, using \eqref{sec-mom} and $\theta<2$, we have 
\setlength\arraycolsep{0.1em}
\begin{eqnarray*}\label{31}
\sum_{\lambda_n\geq S}|r_n|^2&&=2\int_S^\infty\bigg(\sum_{\lambda_n\leq t}
\lambda_n^2|r_n|^2\bigg)\dfrac{dt}{t^3}+\lim_{X\rightarrow\infty}X^{-2}
\bigg(\sum_{\lambda_n\leq X}\lambda_n^2|r_n|^2\bigg)-S^{-2}
\bigg(\sum_{\lambda_n\leq S}\lambda_n^2|r_n|^2\bigg)\nonumber\\
&&\ll \int_S^\infty t^{\theta-3}dt+S^{\theta-2}\ll S^{\theta-2}.
\end{eqnarray*}
By employing this bound, \eqref{lambda}, and Cauchy's inequality we have
$$\sum_{S<\lambda_n\leq T} |r_n| \ll \frac{(T-S)^{\frac{1}{2}} 
(\log{T})^{\frac{1}{2}}}{S^{1-\frac{\theta}{2}}}.$$
Now we choose $\alpha=1/2$, $\beta=1-\theta/2$, and $\gamma=1/2$, and employ
Theorem \ref{lim-dis}. If $1 \le \theta<3-\sqrt{3}$ the conditions given in (b) in Theorem \ref{lim-dis} are satisfied. 
Note that this also implies the result for $0\leq \theta < 1$ since in this case $\sum_{\lambda_n \le T} \lambda_n^2 |r_n|^2 \ll T^{\theta} \le T$.
\end{proof}

\section{Applications of the Main Theorem}
\noindent In this section, by applying Theorem \ref{117}, we prove Corollaries \ref{339}, \ref{132}, and \ref{118}. 
\subsection*{3.1. Proof of Corollary \ref{339}}$~$
\medskip\par
\noindent{\bf  Error term of the prime number theorem for automorphic $\boldsymbol{L}$-functions.}
Let $\piup$ be an irreducible unitary cuspidal automorphic representation of ${\rm GL}_d(\mathbb{A}_\mathbb{Q})$ and let $L(s, \piup)$ be the automorphic $L$-function attached to $\piup$. We follow the notation in the introduction. For $\delta>0$ 
let 
$$\mathbb{C}(\delta)=\mathbb{C}\setminus\{z\in\mathbb{C}\mid
|z+\mu_\piup(j)+2k|\leq\delta,\ 1\leq j\leq d,\ k\geq0\}.$$
We need the following lemma.
\begin{lemma}\label{800}
(i) Let $\sigma\leq -1/2$ then for all $s=\sigma+it\in\mathbb{C}(\delta)$,
$$\dfrac{L'(s, \piup)}{L(s, \piup)}\ll\log{|s|}.$$
(ii) For any integer $m\geq2$, there is $T_m$ with $m\leq T_m\leq m+1$ such that
$$\dfrac{L'(\sigma\pm iT_m, \piup)}{L(\sigma\pm iT_m, \piup)}\ll\log^2{T_m}$$
uniformly for $-2\leq\sigma\leq2$.\\
(iii) For $T\geq2$, the number $N(T, \piup)$ of the zeros of $L(s, \piup)$ in the region $0\leq\Re(s)\leq1$, 
$|\Im(s)|\leq T$ satisfies 
$$N(T+1, \piup)-N(T, \piup)\ll\log{T}$$
and
$$N(T, \piup)\ll T\log{T}.$$
(iv) There is a constant $0\leq\theta<1/2$ such that for all $1\leq j \leq d$,
\begin{equation}
\label{R-bound}
|\alpha_\pi(p, j)|\leq p^\theta ~~\textrm{and}~~|\Re(\mu_\piup(j))|\leq\theta.
\end{equation}
\end{lemma}
\begin{proof}
For (i) see \cite[p. 177]{Mor} for ${\rm GL}_2$ automorphic $L$-functions. The general case is similar. See \cite[ Lemma 4.3(a)(d)]{LY} for (ii) and (iii). For (iv) see \cite[p. 275]{R-S}. Note that in (ii), and (iii) the implied constants depend on $\pi$, and in (i) the implied constant depends on $\delta$ and $\pi$.
\end{proof}
We now establish an explicit formula for $\psi(x, \piup)$.
\begin{prop}\label{801}
Let $\theta$ be the constant given in \eqref{R-bound}. For all $x>1$ and $T\geq 2$ we have 
\begin{equation}\label{802}
\psi(x, \piup)-\delta(x, \piup)=R_\piup-\sum_{|\Im(\rho)|\leq T}\dfrac{x^\rho}{\rho}+
O\Big(\dfrac{x^{\theta+1}\log^2 x}{T}+x^\theta \log{x}+\dfrac{x\log^2{T}}{T\log x}
+\dfrac{x\log{T}}{T}\Big),
\end{equation}
where $\rho$ runs over the nontrivial zeros of $L(s, \piup)$ with $|\Im(\rho)|\leq T$, and
$$R_\piup=\left\{
\begin{array}{lll}
-\dfrac{L'(0, \piup)}{L(0, \piup)}&\text{if}\ L(0, \piup)\neq 0,\\
-\log x-\dfrac{L''(0, \piup)}{2L'(0, \piup)}&\text{if}\ L(0, \piup)=0.
\end{array}
\right.
$$
The implied constant in \eqref{802} depends on $\delta$ in Lemma \ref{800} and $\pi$.
\end{prop}
\begin{proof}
Recall that for $\Re(s)>1$, we have $$-\frac{L^\prime(s, \piup)}{L(s, \piup)}=\sum_{n=1}^{\infty} \frac{\Lambda(n) a_\piup(n)}{n^s}.$$
From \eqref{a_pi} and Lemma \ref{800}(iv) we conclude that $|a_{\pi}(n)|\leq d n^\theta$. 
Let $c=1+1/\log x$, and $T_m$ be as in Lemma \ref{800}(ii). By Perron's formula 
\cite[p. 70, Lemma 3.19]{T} we obtain
\begin{equation}
  \label{psixup}
\psi(x, \piup)=\frac{1}{2\pi i}\int_{c-iT_m}^{c+iT_m}\Big(-\dfrac{L'(s, \piup)}{L(s, \piup)}\Big)\ \dfrac{x^s}{s}ds+
O\Big(\dfrac{x^{\theta+1}\log^2 x}{T}+x^\theta \log{x} \Big).
\end{equation}
Let $U<-1/2$ and $\delta>0$ be such that $U\pm it\in\mathbb{C}(\delta)$ for $t\in [-T_m, T_m]$. Consider the contour which consists of the rectangle 
$\mathcal{C}$ with vertices $c+iT_m,c-iT_m,U+iT_m,U-iT_m$. By the residue theorem, we have
\begin{multline}\label{803}
\frac{1}{2\pi i}\int_{c-iT_m}^{c+iT_m}\Big(-\dfrac{L'(s, \piup)}{L(s, \piup)}\Big)\ \dfrac{x^s}{s}ds=\delta(x, \piup)+R_\piup-
\sum_{\substack{U\leq\Re(\mu)\leq\theta\\|\Im(\mu)|\leq T}}\dfrac{x^\mu}{\mu}-
\sum_{\substack{0\leq\Re(\rho)\leq1\\|\Im(\rho)|\leq T}}\dfrac{x^\rho}{\rho}\\
+\frac{1}{2\pi i}\bigg(\int_{c-iT_m}^{U-iT_m}+\int_{U-iT_m}^{U+iT_m}+\int_{U+iT_m}^{c+iT_m}\bigg)
\Big(-\dfrac{L'(s, \piup)}{L(s, \piup)}\Big)\ \dfrac{x^s}{s}ds,
\end{multline}
where the first and the second sums run over the trivial and the non-trivial zeros of $L(s, \piup)$ inside the 
rectangle $\mathcal{C}$, respectively. If we follow the argument in \cite[pp. 174--178]{Mor} and employ Lemma 
\ref{800}(i) and (ii), we get the following estimates for the integrals on the right-hand side of \eqref{803}. We have
$$\frac{1}{2\pi i}\bigg(\int_{c-iT_m}^{U-iT_m}+\int_{U+iT_m}^{c+iT_m}\bigg)\Big(-\dfrac{L'(s, \piup)}{L(s, \piup)}\Big)\ \dfrac{x^s}{s}ds
\ll \dfrac{x\log^2{T_m}}{T_m\log x},$$
and
$$\frac{1}{2\pi i}\int_{U-iT_m}^{U+iT_m}\Big(-\dfrac{L'(s, \piup)}{L(s, \piup)}\Big)\ \dfrac{x^s}{s}ds \ll\dfrac{T_mx^U\log{|U+iT_m|}}{|U|}.$$
Now we let $U\to-\infty$ through admissible values and note that $(T_mx^U\log{|U+iT_m|})/|U|\to0$. Moreover, Lemma \ref{800}(iv) implies
$$\sum_{\substack{-\infty\leq\Re(\mu)\leq\theta\\|\Im(\mu)|\leq T}}\dfrac{x^\mu}{\mu}\ll
x^\theta\left(1+\sum_{\substack{1\leq j\leq d\\k >0}}\dfrac{x^{-2k}}{\Re(\mu_\piup(j))-2k} \right)\ll x^\theta.$$
Inserting the above estimates in  \eqref{803} together with \eqref{psixup} establishes \eqref{802} in the case
$T=T_m$. Now note that if we change $T_m$ by an arbitrary 
$T\in[m,m+1]$, then we have the same estimate as in \eqref{802}, since by Lemma \ref{800}(iii), we have
$$\sum_{\substack{0\leq\Re(\rho)\leq1\\T_m\leq|\Im(\rho)|\leq T}}\dfrac{x^\rho}{\rho}+
\sum_{\substack{0\leq\Re(\rho)\leq1\\T\leq|\Im(\rho)|\leq T_m}}\dfrac{x^\rho}{\rho}
\ll\dfrac{x\log{T}}{T}.$$
This completes the proof.
\end{proof}
We now show that, under the assumption of the generalized Riemann hypothesis, 
$$E_1(y, \piup)=\dfrac{\psi(e^y, \piup)-\delta(e^y,\piup)}{e^{y/2}}$$
has a limiting distribution.
By pairing the conjugate zeros $\rho=1/2+i\gamma$ and $\bar{\rho}=1/2-i\gamma$ in \eqref{802}, 
for $y>0$ and $X\geq2$, we get
$$E_1(y, \piup)=-2{\rm ord}_{s=1/2}L(s, \piup)+\Re\bigg(\sum_{0<\gamma\leq X}\dfrac{-2e^{i\gamma y}}{\rho}\bigg)+\mathcal{E}_\piup(y,X),$$
where ${\rm ord}_{s=1/2}L(s, \piup)$ is equal to the multiplicity of the zero of $L(s, \piup)$ at $s=1/2$ if $L(1/2, \piup)=0$ 
and ${\rm ord}_{s=1/2}L(s, \piup)=0$ otherwise, and $\mathcal{E}_\piup(y,X)$ satisfies
$$\mathcal{E}_\piup(y,X)=O\bigg(\dfrac{y^2 e^{y(1/2+\theta)}}{X}+ye^{y(\theta-1/2)}+
\dfrac{e^{y/2}\log^2{X}}{yX}+\dfrac{e^{y/2}\log{X}}{X}\bigg).$$
Note that Condition \eqref{9} for $y_0>0$ is satisfied for $\mathcal{E}_\piup(e^y,e^Y)$. 
Setting $r_n=-2/\rho_n$ and $\lambda_n =\Im(\rho_n)$ where the non-trivial zeros of $L(s, \piup)$ are labelled
$(\rho_n)_{n \in \mathbb{N}}$, we obtain from Lemma \ref{800}(iii)
$$ \sum_{\lambda_n \le T} \lambda_n^2 |r_n|^2 = 
\sum_{0<\gamma\leq T}\dfrac{4\gamma^2}{|\rho|^2}\ll\sum_{0<\gamma\leq T}1\ll T\log{T}.$$
Hence, assuming the generalized Riemann hypothesis for $L(s,\piup)$, Corollary \ref{11}(b) implies that $E_1(y, \piup)$ has a limiting distribution.

\subsection*{3.2. Proof of Corollary \ref{132}}
$~$
\medskip\par
Before proceeding we require the following two lemmas.  The first lemma derives an explicit formula for sums of the shape 
$\sum_{n \le x} a_n n^{-\alpha}$.   
\begin{lemma}\label{105}
Let $( a_n )_{n \in \mathbb{N}}$ be a bounded sequence.  Assume there exist complex functions $F(w)$ and $G(w)$ such that
$\sum_{n=1}^{\infty} a_n n^{-w}=\frac{F(w)}{G(w)}$ for $\Re(w) > 1$. 
 Let $x > 1$, $\alpha\in[0,1]$, $\beta\in\mathbb{R}$, $c=1-\alpha+1/\log x$, $b \ne \alpha$, and $b<c+\alpha$. \\
 Assume the following three conditions hold:\\
(i) For any $t>0$, within and on the box $\mathcal{B}_t$ with vertices $c+\alpha+it,b+it,b-it,c+\alpha-it$, $F(s)$ is either holomorphic or it has a simple pole of residue $d_0$ at $s_0 \in (b,c+\alpha)$, and $G(s)$ is holomorphic with simple zeros at $\rho_1,\ldots,\rho_{J}$ inside $\mathcal{B}_t$ and each different from $s_0$. In the case $F(s)$ has a simple pole at $s=s_0$,
let $d_1$ be the value of $F(s)-d_0/(s-s_0)$ at $s=s_0$.  \\
(ii) There exists $\delta=\delta(b)\in\mathbb{R}-\{0\}$ such that
$$\dfrac{F(b+it)}{G(b+it)}=O((|t|+1)^\delta).$$
(iii) There exists an increasing sequence of positive numbers $(T_m)_{m \in \mathbb{N}}$ tending to infinity such that 
$$\dfrac{F(\sigma\pm iT_m)}{G(\sigma \pm iT_m)}=O(T_m^\beta),$$
uniformly on $b\leq\sigma\leq c+\alpha$. \\
Then for $\alpha \in [0,1]$ and $x >1$, 
$$\sum_{n\leq x}\dfrac{a_n}{n^\alpha}=R_{\alpha, s_0}(x)+\sum_{\substack{j=1 \\ \rho_j \in \mathcal{B}_t}}^{J}\dfrac{F(\rho_j)x^{\rho_j-\alpha}}{(\rho_j-\alpha)G'(\rho_j)}+O\Big(\dfrac{x^{1-\alpha}\log x}{T_m}+\dfrac{x^cT_m^{\beta-1}}{\log x}+x^{b-\alpha}(T_m^{\delta}+1)+x^{-\alpha}\Big),$$
where
\begin{equation}
\label{Ralpha}
R_{\alpha, s_0}(x)=\left\{\begin{array}{lll}
0&\quad \text{if} &\quad \alpha, s_0\not\in (b, c+\alpha),\\
\dfrac{F(\alpha)}{G(\alpha)}&\quad\text{if}&\quad\alpha\in (b, c+\alpha), \ s_0\not \in (b,c+\alpha),\\
\dfrac{d_0 x^{s_0-\alpha}}{(s_0-\alpha)G(s_0)}
&\quad\text{if}&\quad\alpha\not\in (b,c+\alpha),\ s_0\in (b, c+\alpha),\\
\dfrac{d_0 x^{s_0-\alpha}}{(s_0-\alpha)G(s_0)}+\dfrac{F(\alpha)}{G(\alpha)}
&\quad\text{if}&\quad\alpha, s_0 \in (b, c+\alpha),\ \alpha\neq s_0,\\
\dfrac{d_0\log x}{G(s_0)}+\dfrac{d_1}{G(s_0)}-\dfrac{d_0G'(s_0)}{G^2(s_0)}
&\quad\text{if}&\quad \alpha, s_0 \in (b, c+\alpha),\ \alpha=s_0.\\
\end{array}\right.
\end{equation}
\end{lemma}
\begin{proof}
Applying Perron's formula \cite[p. 70, Lemma 3.19]{T}  with $c=1-\alpha+1/\log x$ gives
\begin{equation}\label{107}
\sum_{n\leq x}\dfrac{a_n}{n^\alpha}=\dfrac{1}{2\pi i}\int_{c-iT_m}^{c+iT_m}\dfrac{F(s+\alpha)}{G(s+\alpha)}\dfrac{x^s}{s}ds+O\Big(\dfrac{x^{1-\alpha}\log x}{T_m}+x^{-\alpha}\Big).
\end{equation}
If we replace $s$ by $s-\alpha$ in the integral, we obtain
$$\dfrac{1}{2\pi i}\int_{c-iT_m}^{c+iT_m}\dfrac{F(s+\alpha)}{G(s+\alpha)}\dfrac{x^s}{s}ds=\dfrac{1}{2\pi i}\int_{c+\alpha-iT_m}^{c+\alpha+iT_m}\dfrac{F(s)}{G(s)}\dfrac{x^{s-\alpha}}{s-\alpha}ds.$$
Cauchy's residue theorem and (i) imply
\begin{multline}\label{106}
\dfrac{1}{2\pi i}\int_{c+\alpha-iT_m}^{c+\alpha+iT_m}\dfrac{F(s)}{G(s)}\dfrac{x^{s-\alpha}}{s-\alpha}ds=R_{\alpha, s_0}(x)+
\sum_{\substack{j=1 \\ \rho_j \in \mathcal{B}_{T_m}}}^{J}
\dfrac{F(\rho_j)x^{\rho_j-\alpha}}{(\rho_j-\alpha)G'(\rho_j)}\\
+\dfrac{1}{2\pi i}\left(\int_{c+\alpha-iT_m}^{b-iT_m}+\int_{b-iT_m}^{b+iT_m}+\int_{b+iT_m}^{c+\alpha+iT_m}\right)\dfrac{F(s)}{G(s)}\dfrac{x^{s-\alpha}}{s-\alpha}ds,
\end{multline}
where $R_{\alpha, s_0}(x)$  equals the sum of the residues at $s=s_0$ and $s=\alpha$, and the sum appears from the residues at the zeros of $G(s)$.
Taking into account the various cases for $\alpha$ and $s_0$, a simple residue calculation yields \eqref{Ralpha}.  
From assumptions (ii) and (iii) we obtain
\setlength\arraycolsep{0.1em}
\begin{eqnarray}\label{126}
\dfrac{1}{2\pi i}\int_{b-iT_m}^{b+iT_m}\dfrac{F(s)}{G(s)}\dfrac{x^{s-\alpha}}{s-\alpha}ds\ll x^{b-\alpha}(T_m^\delta+1),
\end{eqnarray}
\begin{eqnarray}\label{127}
\dfrac{1}{2\pi i}\int_{b+iT_m}^{c+\alpha+iT_m}\dfrac{F(s)}{G(s)}\dfrac{x^{s-\alpha}}{s-\alpha}ds\ll\dfrac{x^{c}T_m^{\beta-1}}{\log x},
\end{eqnarray}
and similarly 
\begin{eqnarray}\label{128}
\dfrac{1}{2\pi i}\int_{c+\alpha-iT_m}^{b-iT_m}\dfrac{F(s)}{G(s)}\dfrac{x^{s-\alpha}}{s-\alpha}ds\ll\dfrac{x^{c}T_m^{\beta-1}}{\log x}.
\end{eqnarray}
The result follows by combining  \eqref{107}, \eqref{106}, \eqref{126}, \eqref{127}, and \eqref{128}.
\end{proof}
In the previous lemma, a convenient sequence $( T_m )_{m \in \mathbb{N}}$ of reals is chosen so that $F(s)/G(s)$ 
is not too large on the contour $\Im(s)=T_m$.   Consequently, in the explicit formula for $\sum_{n \le x} a_n n^{-\alpha}$, 
the sum over $\rho_j$ is constrained by the condition $|\Im(\rho_j)| \le T_m$. 
 The next lemma allows us to replace this condition by $|\Im(\rho_j)| \le T$ for any $T \ge 1$.  
\begin{lemma}\label{102}
Let $(z_n)_{n \in \mathbb{N}} \subset\mathbb{C}$ and $( \lambda_n )_{n \in \mathbb{N}}  \subset\mathbb{R}^{+}$ be sequences and 
let $x,c_1$, and $c_2$ be positive reals. 
Let $T,T' \in [1,\infty)$  such that $|T-T'| \le 1$.  Assume that for $t \ge 1$ we have 
\begin{equation}\label{104}
\sum_{\lambda_n\leq t}|z_n|^2\ll t^{c_1}
\end{equation}
and
\begin{equation}\label{103}
\sum_{t < \lambda_n \le t+1}1\ll (\log t)^{c_2}.
\end{equation}
Then 
\begin{equation*}\label{400}
\sum_{\lambda_n\leq T'}\dfrac{z_nx^{\frac{1}{2}+i\lambda_n}}{\frac{1}{2}+i\lambda_n}=\sum_{\lambda_n\leq T}\dfrac{z_nx^{\frac{1}{2}+i\lambda_n}}{\frac{1}{2}+i\lambda_n}+O\Big(x^{\frac{1}{2}} T^{(c_1-2)/2}(\log T)^{c_2/2}\Big).
\end{equation*}
\end{lemma}
\begin{proof}
We begin by assuming $T-1 \le T' \leq T$. By the Cauchy-Schwarz inequality 
\begin{equation}\label{129}
\bigg|\sum_{T'<\lambda_n\leq T}\dfrac{z_nx^{\frac{1}{2}+i\lambda_n}}{\frac{1}{2}+i\lambda_n}\bigg|\leq x^{\frac{1}{2}}\bigg(\sum_{T'<\lambda_n\leq T}\bigg|\dfrac{z_n}{\frac{1}{2}+i\lambda_n}\bigg|^2\bigg)^{1/2}\bigg(\sum_{T'<\lambda_n\leq T}1\bigg)^{1/2}
 \ll x^{\frac{1}{2}} T^{\frac{c_1-2}{2}} (\log T)^{c_2/2}
\end{equation}
by  \eqref{104} and  \eqref{103}.  In the case $T < T' \le T+1$, we obtain the same bound as  \eqref{129}. 
\end{proof}
\noindent We now prove Corollary \ref{132}.  In each part of this corollary, we shall apply Corollary \ref{11}(b) to establish the 
existence of the limiting distribution.  
\medskip\par
\noindent{\bf (i) Weighted Sums of the M\"{o}bius Function.}
In this proof we assume the Riemann hypothesis and assumption \eqref{330}.  We 
shall show that $E_2(y,\alpha)$, defined in  \eqref{115}, possesses a limiting distribution. 
We start by establishing an explicit formula for 
$$M_\alpha(x)=\sum_{n\leq x}\dfrac{\mu(n)}{n^\alpha}.$$
We first consider the case $\alpha\neq0$. Let $0< b<\min(1/2,\alpha)$ and $0<\varepsilon< 1/2-b$. 
Under the assumption of the Riemann hypothesis, 
there exists a sequence $(T_m )_{m \in \mathbb{N}}$, where $T_m\in[m-1,m]$, such that 
\begin{equation} \label{Tm}
  |\zeta(\sigma+iT_m)|^{-1} \ll T_m^\varepsilon
\end{equation}
uniformly for $-1\leq\sigma\leq2$ (see \cite[pp. 357--358]{T}). Moreover, for any $\varepsilon>0$, we have 
\[
  |\zeta(b+it)|^{-1} \ll |t|^{-1/2+b+\varepsilon}
\]  
for $|t|\geq1$ (see \cite[Corollary 10.5 and Theorems 13.18 and 13.23]{MontVan}).   
By taking $F(s)=1$, $G(s)=\zeta(s)$, $a_n=\mu(n)$, $\beta=\varepsilon$, and $\delta=-1/2+b+\varepsilon$ in Lemma \ref{105} we derive
\begin{equation}\label{109}
M_\alpha(x)= \frac{1}{\zeta(\alpha)}+\sum_{|\gamma|\leq T_m}\dfrac{x^{\rho-\alpha}}{(\rho-\alpha)\zeta'(\rho)}+O\left(\dfrac{x^{1-\alpha}\log x}{T_m}+\dfrac{x^{1-\alpha}}{T_m^{1-\varepsilon}\log x}+x^{b-\alpha} \right),
\end{equation}
where $\rho$ ranges over the non-trivial zeros of $\zeta(s)$. 
Let $T \ge 1$ and $m \ge 1$ be the natural number  such that $T\in[m-1,m]$. 
Label the non-trivial zeros of $\zeta(s)$ with positive imaginary part in non-decreasing order by $( \rho_n )_{n \in \mathbb{N}}$. 
An application of Lemma \ref{102} with  $\lambda_n=\Im(\rho_n)$, $z_n=\zeta'(\rho_n)^{-1}$,
$c_1=\theta$, $c_2=1$, $\beta=1/2$, $T$, and $T'=T_m$ implies that
\begin{equation}\label{134}
\sum_{|\gamma|\leq T_m}\dfrac{x^{\rho-\alpha}}{(\rho-\alpha)\zeta'(\rho)}=\sum_{|\gamma|\leq T}\dfrac{x^{\rho-\alpha}}{(\rho-\alpha)\zeta'(\rho)}+O\big(x^{1/2-\alpha}T^{(\theta-2)/2}(\log T)^{1/2}\big).
\end{equation}
Substituting \eqref{134} in \eqref{109}, for $\alpha\neq0$, we have
\begin{multline}\label{108}
M_\alpha(x)= \frac{1}{\zeta(\alpha)}+\sum_{|\gamma|\leq T}\dfrac{x^{\rho-\alpha}}{(\rho-\alpha)\zeta'(\rho)}\\+O\left(\dfrac{x^{1-\alpha}\log x}{T}+\dfrac{x^{1-\alpha}}{T^{1-\varepsilon}\log x}+x^{1/2-\alpha}\left(T^{\theta-2}\log T\right)^{1/2}+x^{b-\alpha}   \right),
\end{multline}
valid for $x >1$ and $T \ge 1$. If $\alpha=0$, we let $0<b<1/2$. Then similarly we have
\begin{equation}
\label{M0}
M_0(x)=\sum_{|\gamma|\leq T}\dfrac{x^{\rho}}{\rho\zeta'(\rho)}
+O\left(\dfrac{x\log x}{T}+\dfrac{x}{T^{1-\varepsilon}\log x}
+x^{1/2}\left(T^{\theta-2}\log T\right)^{1/2}+x^b   \right).
\end{equation}
We now analyze $E_2(y,\alpha)$ in the cases $\alpha \in (0,1/2)$, $\alpha \in (1/2,1]$, $\alpha =0$, and $\alpha =1/2$. 

For $0<\alpha<1/2$, by \eqref{108}, for $X\geq 1$ and $y>0$,  we have
\begin{multline*}
E_2(y,\alpha)=\dfrac{1}{e^{y(1/2-\alpha)}\zeta(\alpha)}+e^{y(-1/2+\alpha)}\sum_{|\gamma|\leq X}\dfrac{e^{y(\rho-\alpha)}}{(\rho-\alpha)\zeta'(\rho)}\\
+O\left(\dfrac{ye^{y/2}}{X}+\dfrac{e^{y/2}}{yX^{1-\varepsilon}}+\left(X^{\theta-2}\log X\right)^{1/2}+\dfrac{1}{e^{y(1/2-b)}}\right).
\end{multline*}
Thus
$$E_2(y,\alpha)=\Re\bigg(\sum_{0<\gamma\leq X}\dfrac{2e^{iy\gamma}}{(\rho-\alpha)\zeta'(\rho)}\bigg)+\mathcal{E}_{\mu,\alpha}(y,X),$$
where 
$$\mathcal{E}_{\mu,\alpha}(y,X)=O\left(\dfrac{ye^{y/2}}{X}+\dfrac{e^{y/2}}{yX^{1-\varepsilon}}+\left(X^{\theta-2}\log X\right)^{1/2}+\dfrac{1}{e^{y(1/2-\alpha)}}\right).$$
Note that in this case the term $e^{y(\alpha-1/2)}$ in $\mathcal{E}_{\mu,\alpha}(y,X)$ comes from the 
term $e^{y(\alpha-1/2)}/\zeta(\alpha)$ in $E_2(y,\alpha)$, since we chose $b<\alpha$.

For $\frac{1}{2} < \alpha \le 1$, we recall
that $E_2(y,\alpha)=e^{y(-1/2+\alpha)}\big(M_\alpha(e^y)-1/\zeta(\alpha)\big)$. 
By  \eqref{108} and by pairing conjugate zeros, we obtain  
\begin{align*}
E_2(y,\alpha)
&=\Re\bigg(\sum_{0<\gamma\leq X}\dfrac{2e^{iy\gamma}}{(\rho-\alpha)\zeta'(\rho)}\bigg)+\mathcal{E}_{\mu,\alpha}(y,X),
\end{align*}
for $X \ge 1$ and $y > 0$, where 
\begin{equation}\label{55}
\mathcal{E}_{\mu,\alpha}(y,X)=O\left(\dfrac{ye^{y/2}}{X}+\dfrac{e^{y/2}}{yX^{1-\varepsilon}}+\left(X^{\theta-2}\log X\right)^{1/2}+\dfrac{1}{e^{y(1/2-b)}}\right).
\end{equation}

For $\alpha=0$, from \eqref{M0} we have
$$E_2(y,0)= \Re\bigg(\sum_{0<\gamma\leq X}\dfrac{2e^{iy\gamma}}{\rho\zeta'(\rho)}\bigg)+\mathcal{E}_{\mu,0}(y,X),$$
where $\mathcal{E}_{\mu,0}(y,X)$ satisfies \eqref{55}.

Finally, for $\alpha=1/2$, \eqref{108} implies
$$E_2(y,1/2)= \frac{1}{\zeta(\tfrac{1}{2})}+\Re\bigg(\sum_{0<\gamma\leq X}\dfrac{2e^{iy\gamma}}{(\rho-1/2)\zeta'(\rho)}\bigg)+\mathcal{E}_{\mu,1/2}(y,X),$$
where $\mathcal{E}_{\mu,1/2}(y,X)$ is bounded as \eqref{55}. 

Note that $\mathcal{E}_{\mu,\alpha}(y,e^Y)$ satisfies \eqref{9} for $y_0>0$, for any $\alpha\in[0,1]$.
Setting $r_n=2/(\rho_n-\alpha) \zeta'(\rho)$ and $\lambda_n =\Im(\rho_n)$, it follows from
\eqref{330} that
$$ \sum_{\lambda_n \le T} \lambda_n^2 |r_n|^2 =
\sum_{0<\gamma\leq T}\dfrac{4\gamma^2}{|(\rho-\alpha)\zeta'(\rho)|^2}\ll T^\theta.$$
Thus Corollary \ref{11}(b) implies that, under the assumptions of the Riemann hypothesis for $\zeta(s)$ and \eqref{330}, $E_2(y,\alpha)$ has a limiting distribution.
\medskip\par
\noindent{(\bf ii) Weighted Sums of the Liouville Function.}
In this part, we show that $E_3(y,\alpha)$, defined by \eqref{58}, possesses a limiting distribution. 
We begin by establishing an explicit formula for $L_\alpha(x)=\sum_{n \le x} \lambda(n) n^{-\alpha}$.
Assume the Riemann hypothesis for $\zeta(s)$ and \eqref{330}. 
For $\alpha\in(0,1]$ and $x>1$, let
\begin{equation*}
R_{\alpha, s_0}(x)=\left\{
\begin{array}{lll}
\frac{x^{1/2-\alpha}}{(1-2\alpha)\zeta(1/2)}+\frac{\zeta(2\alpha)}{\zeta(\alpha)}&\quad\mbox{if}&\quad\alpha\neq1/2,\\
\frac{\log x}{2\zeta(1/2)}+\frac{\gamma_0}{\zeta(1/2)}-\frac{\zeta^\prime (1/2)}{2\zeta(1/2)^2}&\quad\mbox{if}&\quad\alpha=1/2,
\end{array}
\right.
\end{equation*}
where $\gamma_0$ is Euler's constant. Let $0<\epsilon<b<\min(1/4,\alpha)$.  Then we have 
\[
\Big|\frac{\zeta(2(b+it))}{\zeta(b+it)} \Big| \ll|t|^{-b+\varepsilon}
\]
 for all $|t|\geq1$, and 
 \[
  \Big| \frac{\zeta(2(\sigma+iT_m))}{\zeta(\sigma+iT_m)} \Big| \ll T_m^{1/2-2b+\varepsilon}
 \]
  uniformly for $b\leq\sigma\leq c+\alpha$ (see \cite[Corollary 10.5 and Theorems 13.18 and 13.23]{MontVan}), 
  where $(T_m)_{m \in \mathbb{N}}$ is the sequence introduced in \eqref{Tm}. Set $F(s)=\zeta(2s)$, $G(s)=\zeta(s)$, $z_m=\lambda(m)$, $\beta=1/2-2b+\varepsilon$, and $\delta=-b+\varepsilon$. 
Then if $\alpha\neq0$, Lemmas \ref{105} and \ref{102} imply that, for $x>1$ and $T\geq 1$,
\begin{multline}\label{110}
L_\alpha(x)=R_{\alpha, s_0}(x)+\sum_{|\gamma|\leq T}\dfrac{x^{\rho-\alpha}}{\rho-\alpha}\dfrac{\zeta(2\rho)}{\zeta'(\rho)}\\+O\left(\dfrac{x^{1-\alpha}\log x}{T}+\dfrac{x^{1-\alpha}T^{-1/2-2b+\varepsilon}}{\log x}+x^{1/2-\alpha}\left(T^{\theta-2}\log T\right)^{1/2}+x^{b-\alpha}\right).
\end{multline}
If $\alpha=0$, we let $0<\epsilon<b<1/4$. Similarly, we have 
\begin{equation}\label{1000}
L_0(x)=\frac{x^{{1}/{2}}}{\zeta(1/2)}+\sum_{|\gamma|\leq T}\dfrac{x^{\rho}}{\rho}\dfrac{\zeta(2\rho)}{\zeta'(\rho)}
+O\left(\dfrac{x\log x}{T}+\dfrac{xT^{-1/2-2b+\varepsilon}}{\log x}+x^{1/2}
\left(T^{\theta-2}\log T\right)^{1/2}+x^{b}\right).
\end{equation}
For $\alpha\in[0,1]$, let
\begin{equation*}\label{59}
C_\alpha=\left\{
\begin{array}{lll}
\frac{1}{(1-2\alpha)\zeta(1/2)}&\quad\text{if} & \quad0\leq\alpha<1/2\ \text{or}\ 1/2<\alpha\leq1,\\
\frac{\gamma_0}{\zeta(1/2)}-\frac{\zeta^\prime(1/2)}{2\zeta(1/2)^2}&\quad\text{if} & \quad\alpha=1/2.
\end{array}
\right.
\end{equation*}
Then \eqref{110} and \eqref{1000} imply that, for $y>0$ and $X\geq 1$,
\begin{eqnarray*}\label{60}
E_3(y,\alpha)&=&C_\alpha+e^{y(-1/2+\alpha)}\sum_{|\gamma|\leq X}\dfrac{\zeta(2\rho)e^{y(\rho-\alpha)}}{(\rho-\alpha)\zeta'(\rho)}+\mathcal{E}_{\lambda,\alpha}(y,X)\nonumber\\
&=&C_\alpha+\sum_{|\gamma|\leq X}\dfrac{\zeta(2\rho)e^{iy\gamma}}{(\rho-\alpha)\zeta'(\rho)}+\mathcal{E}_{\lambda,\alpha}(y,X)\nonumber\\
&=&C_\alpha+\Re\bigg(\sum_{0<\gamma\leq X}\dfrac{2\zeta(2\rho)e^{iy\gamma}}{(\rho-\alpha)\zeta'(\rho)}\bigg)+\mathcal{E}_{\lambda,\alpha}(y,X),
\end{eqnarray*}
where
$$\mathcal{E}_{\lambda,\alpha}(y,X)\ll\dfrac{ye^{y/2}}{X}+\dfrac{e^{y/2}X^{-1/2-2b+\varepsilon}}{y}
+\left(X^{\theta-2}\log X\right)^{1/2}+\dfrac{1}{e^{y(1/2-b)}}.$$
Observe that \eqref{9} for $y_0>0$ holds for $\mathcal{E}_{\lambda,\alpha}(y,e^Y)$.
Since $r_n=2 \zeta(2 \rho_n)/(\rho_n-\alpha) \zeta'(\rho)$ and $\lambda_n =\Im(\rho_n)$, it follows from
\eqref{330} that
$$
\sum_{\lambda_n \le T} \lambda_n^2 |r_n|^2=
\sum_{0<\gamma\leq T}\dfrac{4\gamma^2|\zeta(2\rho)|^2}{|(\rho-\alpha)\zeta'(\rho)|^2}\ll \sum_{0<\gamma\leq T}\dfrac{4\gamma^2(\log \gamma)^{3/2+\varepsilon}}{|(\rho-\alpha)\zeta'(\rho)|^2}\ll T^\theta(\log T)^{3/2+\varepsilon}.$$
Note that, in the previous inequalities we have used the fact that $\zeta(1+it)=O\big((\log t)^{3/4+\varepsilon}\big)$ (see \cite[Theorem 6.14]{T}). Hence by Corollary \ref{11}(b), under the assumptions of the Riemann hypothesis for $\zeta(s)$ and \eqref{330}, $E_3(y,\alpha)$ has a limiting distribution.
\medskip\par
\noindent{\bf (iii) The Summatory Function of the M\"{o}bius Function in Arithmetic Progressions.}
In this part we prove the existence of a limiting distribution for $E_4(y;q,a)$ defined in \eqref{130}. We first establish an explicit formula for 
$$M(x;q,a):=\sum_{\substack{n \le x \\ n \equiv a \text{ mod } q}}  \mu(n),$$
where $q\geq 2$ and $(a,q)=1$. Let $0<b<1/2$ and $0<\varepsilon < 1/2-b$. Assume the generalized 
Riemann hypothesis for Dirichlet $L$-functions modulo $q$ and \eqref{135}.
An argument analogous to the proof of the existence of the sequence $(T_m)_{m \in \mathbb{N}}$ introduced in \eqref{Tm} may be carried out for Dirichlet $L$-functions. Following the proof of \cite[Theorem 13.22]{MontVan}, we are able to show that the generalized Riemann hypothesis for Dirichlet $L$-functions implies that there is a sequence $(T_{m,\chi})_{m \in \mathbb{N}}$, where $T_{m,\chi}\in[m-1,m]$, such that 
$$|L(\sigma+iT_{m,\chi},\chi)|^{-1} \ll T_{m,\chi}^\varepsilon
$$ 
uniformly for $-1\leq \sigma \leq 2$.
Moreover, for any $\varepsilon>0$, we have 
$$
|L(b+it,\chi)|^{-1}\ll|t|^{-1/2+b+\varepsilon}$$
(see \cite[Corollary 10.10 and p. 445, Exercises 8 and 10]{MontVan}).  The orthogonality relation for characters asserts that
$$\dfrac{1}{\varphi(q)}\sum_{\chi\;\mathrm{mod}\;q}\overline{\chi(a)}\chi(n)
=\left\{\begin{array}{ll}1&\quad\text{if}\ n\equiv a\ (\text{mod}\ q),\\0&\quad\text{otherwise,}\end{array}\right.$$
(see \cite[p. 122]{MontVan}). Thus
\begin{equation}\label{111}
M(x;q,a)=\dfrac{1}{\varphi(q)}\sum_{\chi\;\mathrm{mod}\;q}\overline{\chi(a)}\sum_{n\leq x}\mu(n)\chi(n).
\end{equation}
Let $F(s)=1$, $G(s)=L(s,\chi)$, $z_n=\mu(n)\chi(n)$, $\beta=\varepsilon$, $\delta=-1/2+b+\varepsilon$, and $\alpha=0$. Then by applying 
a slight variant of
Lemma \ref{105} which takes into the consideration the potential pole of $1/L(s, \chi)$ at $s=1/2$ and Lemma \ref{102}, we obtain, for $x>1$ and $T \ge 1$,
\setlength\arraycolsep{0.1em}
\begin{eqnarray*}
\sum_{n\leq x}\mu(n)\chi(n)&=&\mathrm{Res}_{s=\frac{1}{2}} \Big( 
\frac{x^s}{L(s, \chi) s} \Big)+\sum_{\substack{{|\gamma_\chi|\leq T} \\ {\gamma_\chi\neq 0}}}\dfrac{x^{\rho_\chi}}{\rho_\chi L'(\rho_\chi,\chi)}\\
&&+O\left(\dfrac{x\log x}{T}+\dfrac{x}{T^{1-\varepsilon}\log x}+x^{1/2}(T^{\theta-2}\log T)^{1/2}+x^b  \right),
\end{eqnarray*}
where $\mathrm{Res}_{s=\frac{1}{2}}(.)$ denote the residue at $s=1/2$.
Substituting this in \eqref{111} implies that, for $x>1$ and $T \ge 1$,
\begin{eqnarray*}
\label{333}
M(x;q,a)&=&\frac{1}{\varphi(q)}\sum_{\substack{ \chi ~\text{mod}~ q \\ L(1/2,\chi)=0}}
\overline{\chi(a)} \mathrm{Res}_{s=\frac{1}{2}} \Big( 
\frac{x^s}{L(s, \chi) s} \Big)+\dfrac{1}{\varphi(q)}\sum_{\chi\;\mathrm{mod}\;q}\overline{\chi(a)}\sum_{\substack{{|\gamma_\chi|\leq T}\\{\gamma_\chi\neq 0}}}\dfrac{x^{\rho_\chi}}{\rho_\chi L'(\rho_\chi,\chi)}\\
&&+O\left(\dfrac{x\log x}{T}+\dfrac{x}{T^{1-\varepsilon}\log x}+x^{1/2}(T^{\theta-2}\log T)^{1/2}+x^b    \right).
\end{eqnarray*}
Assuming the generalized Riemann hypothesis for Dirichlet $L$-functions modulo $q$ and \eqref{135}, it follows that, for $y>0$ and $X\geq 1$, 
\setlength\arraycolsep{0.1em}
\begin{eqnarray}\label{500}
E_4(y;q,a)&=&\frac{1}{\varphi(q)}\sum_{\substack{ \chi ~\text{mod}~ q \\ L(1/2,\chi)=0}}
\overline{\chi(a)} \mathrm{Res}_{s=\frac{1}{2}} \Big( 
\frac{e^{ys}}{L(s, \chi) s} \Big)\nonumber \\
&&+\dfrac{1}{\varphi(q)}\sum_{\chi\;\mathrm{mod}\;q}\overline{\chi(a)}\sum_{\substack{{|\gamma_\chi|\leq X}\\{\gamma_\chi\neq 0}}}\dfrac{e^{iy\gamma_\chi}}{\rho_\chi L'(\rho_\chi,\chi)}
+\mathcal{E}_{\mu,q,a}(y,X),
\end{eqnarray}
where 
\setlength\arraycolsep{0.1em}
\begin{eqnarray*}
\mathcal{E}_{\mu,q,a}(y,X)\ll\dfrac{ye^{y/2}}{X}+\dfrac{e^{y/2}}{yX^{1-\varepsilon}}+\dfrac{(\log X)^{1/2}}{X^{1-\theta/2}}+\dfrac{1}{e^{y(1/2-b)}}.
\end{eqnarray*}
Let $(\lambda_n)_{n\in \mathbb{N}}$ be the non-decreasing sequence that consists of all the numbers $\gamma_\chi>0$ satisfying $L(1/2+i\gamma_\chi,\chi)=0$, for some Dirichlet character $\chi$ mod $q$, and let $(r_n)_{n\in \mathbb{N}}$ be defined as
$$r_n=\dfrac{2\;\overline{\chi_{\lambda_n}(a)}}{\varphi(q)(1/2+i\lambda_n)L'(1/2+i\lambda_n,\chi_{\lambda_n})},$$ 
where $\chi_{\lambda_n}$ is the character which corresponds to $\lambda_n$. We can rewrite \eqref{500} in the form of
$$E_4(y;q,a)=\frac{1}{\varphi(q)}\sum_{\substack{ \chi ~\text{mod}~ q \\ L(1/2,\chi)=0}}
\overline{\chi(a)} \mathrm{Res}_{s=\frac{1}{2}} \Big( 
\frac{e^{ys}}{L(s, \chi) s} \Big)+\Re\bigg(\sum_{\lambda_n<X}r_ne^{iy\lambda_n}\bigg)+\mathcal{E}_{\mu,q,a}(y,X).$$
Observe that \eqref{9} for $y_0>0$ holds for $\mathcal{E}_{\mu,q,a}(y,e^Y)$ and \eqref{135} implies
$$\sum_{\lambda_n\leq T}\lambda_n^2|r_n|^2\ll T^\theta,$$
for $1\leq \theta<3-\sqrt{3}$. Hence Corollary \ref{11}(b) implies that, under the assumptions of the generalized Riemann hypothesis for Dirichlet $L$-functions modulo $q$ and \eqref{135}, $E_4(y;q,a)$ has a limiting distribution.

\subsection*{3.2. Proof of Corollary \ref{118}}$~$
\medskip\par
\noindent{\bf Chebotarev's Density Theorem.}
Let $K/k$ be a normal extension of number fields with corresponding Galois group $G$. We shall consider the squaring function $\mathrm{sq}: G \to G$  given by $\mathrm{sq}(x)=x^2$.
For a conjugacy class $C$ of $G$, let $A_1,\ldots,A_t$ be the conjugacy classes which satisfy $A_i^2\subseteq C$. 
We observe that 
$$\mathrm{sq}^{-1}(C)=\bigcup_{i=1}^tA_i$$
and define 
$$c(G,C)=-1+\frac{|\mathrm{sq}^{-1}(C)|}{|C|}+2\sum_{\chi \ne \chi_0}\overline{\chi(C)}\mathrm{ord}_{s=1/2} L(s,\chi,K/k)$$
where $\chi$ ranges over the irreducible characters of $G$ and $\chi_0$ denotes the trivial character. 
It was proven in  \cite[pp. 71--73]{NgPhd} that 
the generalized Riemann hypothesis and Artin's holomorphy conjecture imply 
that, for $x>1$, $T \ge 1$ and $1\leq j\leq r$,
\begin{multline}\label{100}
\frac{\log x}{\sqrt{x}}\left( \frac{|G|}{|C_j|}\pi_{C_j}(x) - \pi_{k}(x) \right)=\\-c(G,C_j)-\sum_{\chi \ne \chi_0}\overline{\chi(C_j)}\bigg( \sum_{0<|\gamma_{\chi}| \le T}\frac{ x^{i\gamma_{\chi}}}{1/2 + i\gamma_{\chi}}\bigg) + O \left( \frac{x^{1/2} \log^{2}(xT)}{T} +  \frac{1}{\log x} \right),
\end{multline}
where for each $\chi$, $\rho_\chi=1/2+i\gamma_\chi$ 
runs over the non-trivial zeros 
of $L(s,\chi,K/k)$. 
In this formula, the term $c(G,C_j)$ is the number field analogue of the constant term $c(q,a)$ which appears in the Chebyshev bias phenomenon.  
Let $(\lambda_n)_{n \in \mathbb{N}}$ be the non-decreasing sequence that consists of all the numbers $\gamma_\chi>0$ which satisfy $L(1/2+i\gamma_\chi,\chi,K/k)=0$ for some $\chi \ne \chi_0$. 
Suppose that $\chi_n$ is the character which corresponds to $\lambda_n$, and for $1\leq j\leq r$ set 
$r_{j,n}=-2\overline{\chi_n(C_j)}/(1/2+i\lambda_n)$.
Then \eqref{100} implies that
$$E_5^{(j)}(y):=\left(\frac{|G|}{|C_j|}\pi_{C_j}(e^y) - \pi_{k}(e^y)\right)ye^{-y/2}=-c(G,C_j)+\Re\bigg(\sum_{0<\lambda_n\leq X}r_{j,n}e^{iy\lambda_n}\bigg)+\mathcal{E}_{G;C_j}(y,X),$$ 
where 
$$\mathcal{E}_{G;C_j}(y,X)=O\bigg(\frac{e^{y/2} \log^{2}(e^yX)}{X} + \frac{1}{y}\bigg).$$
Observe that Condition \eqref{9} for $y_0>0$ holds for $\mathcal{E}_{G;C_j}(y,e^Y)$ and by \cite[Theorem 5.8]{Iwaniec} we have
$$\sum_{\lambda_n\leq T}\lambda_n^2|r_n|^2\ll\sum_{\lambda_n\leq T}1\ll T\log T.$$
Therefore, Theorem \ref{117} implies that, under the assumptions of generalized Riemann hypothesis and Artin's holomorphy conjecture, 
$\vec{E}_5(y)=\big(E_5^{(1)}(y),\ldots,E_5^{(r)}(y)\big)$ has a limiting distribution. \\

\section{Calculation of the Fourier transform $\widehat{\mu}$}
\begin{proof}[Proof of Theorem \ref{66}]
Let $\vec{r}_m=\big(r_1(\lambda_m),\ldots,r_\ell(\lambda_m)\big)$ and $N\in\mathbb{N}$. 
By Proposition \ref{36}, the vector-valued function
$$\vec{P}(y)=\bigg(c_1+\Re\Big(\sum_{m=1}^Nr_1(\lambda_m)e^{iy\lambda_m}\Big),\ldots,c_\ell+\Re\Big(\sum_{m=1}^Nr_\ell(\lambda_m)e^{iy\lambda_m}\Big)\bigg)$$
has a limiting distribution $\mu_N$. Since $\{\lambda_1, \cdots, \lambda_N\}$ is linearly independent then by the  
Kronecker-Weyl theorem \cite[Chapter 1]{Hlawka} we have
$$\lim_{Y\rightarrow \infty} \frac{1}{Y} \int_{0}^{Y} g\left(\frac{y\lambda_1}{2\pi}, \cdots, \frac{y\lambda_N}{2\pi}\right) dy =\int_{\mathbb{T}^N} g(a) d\omega,$$
where $g: \mathbb{R}^N \rightarrow \mathbb{R}$ is any continuous function of period $1$ in each of its variables and 
$d\omega(\theta_1,\dots,\theta_N)$ is the normalized Haar measure on $\mathbb{T}^N$ which is equal to the Lebesgue measure $d\theta_1\dots d\theta_N$ on $\mathbb{T}^N$. Hence, 
by taking $f(t_1,\ldots,t_\ell)=\exp(-i{\scriptstyle\sum_{k=1}^\ell\xi_k t_k})$ and $A=\mathbb{T}^N$ in \eqref{348}, we obtain
\setlength\arraycolsep{0.1em}
\begin{eqnarray}\label{F-T:1}
&&\int_{\mathbb{R}^\ell}e^{-i{\scriptstyle\sum_{k=1}^\ell\xi_k t_k}}d\mu_N(t_1,\ldots,t_\ell)\nonumber\\
&&\quad\quad\quad\quad\quad\quad=\int_{\mathbb{T}^N}\exp\bigg(-i{\textstyle\sum_{k=1}^\ell}\Big[c_k+\Re\Big({\textstyle\sum_{m=1}^N}r_k(\lambda_m)e^{2\pi i\theta_m}\Big)\Big]\xi_k\bigg)d\omega(\theta_1,\ldots,\theta_N)\nonumber\\
&&\quad\quad\quad\quad\quad\quad=e^{-i{\scriptstyle \sum_{k=1}^\ell c_k\xi_k}}\int_{\mathbb{T}^N}\exp\bigg(-i\Re\Big({\textstyle\sum_{m=1}^N}\big(\vec{r}_m\cdot\vec{\xi}\big)e^{2\pi i\theta_m}\Big)\bigg)d\theta_1\dots d\theta_N\nonumber\\
&&\quad\quad\quad\quad\quad\quad=e^{-i{\scriptstyle \sum_{k=1}^\ell c_k\xi_k}}\times\prod_{m=1}^N\int_0^1\exp\left(-i\Re\left(\big(\vec{r}_m\cdot\vec{\xi}\big)e^{2\pi i\theta}\right)\right)d\theta.
\end{eqnarray}
Thus, in view of \eqref{353} and \eqref{F-T:1} we deduce that
\setlength\arraycolsep{0.1em}
\begin{eqnarray*}
\hat{\mu}(\vec{\xi})=\int_{\mathbb{R}^\ell}e^{-i{\scriptstyle\sum_{k=1}^\ell\xi_k t_k}}d\mu(t_1,\ldots,t_\ell)&&=\lim_{N\to\infty}\int_{\mathbb{R}^\ell}e^{-i{\scriptstyle\sum_{k=1}^\ell\xi_k t_k}}d\mu_N(t_1,\ldots,t_\ell)\\
&&=e^{-i{\scriptstyle \sum_{k=1}^\ell c_k\xi_k}}\times\prod_{m=1}^\infty\int_0^1\exp\left(-i\Re\left(\big(\vec{r}_m\cdot\vec{\xi}\big)e^{2\pi i\theta}\right)\right)d\theta.
\end{eqnarray*}
If $\vec{r}_m\cdot\vec{\xi}\neq0$, then 
\setlength\arraycolsep{0.1em}
\begin{eqnarray}\label{F-T:2}
\int_0^1\exp\left(-i\Re\left(\big(\vec{r}_m\cdot\vec{\xi}\big)e^{2\pi i\theta}\right)\right)d\theta&&=\int_0^1\exp\left(-i\Re\left(|\vec{r}_m\cdot\vec{\xi}|e^{i(2\pi\theta+\mathrm{arg}(\vec{r}_m\cdot\vec{\xi}))}\right)\right)d\theta\nonumber\\
&&=\int_0^1\exp\left(-i|\vec{r}_m\cdot\vec{\xi}|\cos\big(2\pi\theta+\mathrm{arg}\big(\vec{r}_m\cdot\vec{\xi}\big)\big)\right)d\theta\nonumber\\
&&=\int_{\mathrm{arg}(\vec{r}_m\cdot\vec{\xi})/2\pi}^{1+\mathrm{arg}(\vec{r}_m\cdot\vec{\xi})/2\pi}\exp\big(-i|\vec{r}_m\cdot\vec{\xi}|\cos(2\pi t)\big)dt\nonumber\\
&&=\int_0^1\exp\big(-i|\vec{r}_m\cdot\vec{\xi}|\cos(2\pi t)\big)dt\nonumber\\
&&=J_0\big(\big|{\textstyle\sum_{k=1}^\ell}r_k(\lambda_m)\xi_k\big|\big).
\end{eqnarray}
If $\vec{r}_m\cdot\vec{\xi}=0$, then \eqref{F-T:2} holds trivially. Hence
\begin{equation*}
\hat{\mu}(\vec{\xi})=e^{-i{\scriptstyle \sum_{k=1}^\ell c_k\xi_k}}\times\prod_{m=1}^\infty J_0\big(\big|{\textstyle\sum_{k=1}^\ell}r_k(\lambda_m)\xi_k\big|\big).
\end{equation*}
\end{proof}

\medskip\par

\noindent {\bf Acknowledgements}.   The authors thank Daniel Fiorilli, Peter Humphries, and Youness Lamzouri for their comments on  this article. 
We are also grateful to Abbas Momeni for providing us with the reference for the proof of Lemma \ref{402}.

\end{document}